\documentclass[10pt,a4paper]{article}
\usepackage{amsmath,amssymb,amsthm,amsfonts,graphicx}

\usepackage{float}
\usepackage[section]{placeins}

\newtheorem{theorem}{Theorem}[section]
\newtheorem{lemma}[theorem]{Lemma}

\theoremstyle{definition}
\newtheorem{definition}[theorem]{Definition}
\newtheorem{remark}[theorem]{Remark}

\newcommand{\gu}{\mathrm{u}}
\newcommand{\gv}{\mathrm{v}}
\newcommand{\gw}{\mathrm{w}}

\usepackage{tikz}
\usetikzlibrary{snakes}
\usepackage{xcolor}
\def\mod{\mathop\mathrm{mod}\nolimits}

\newcommand{\Real}{\mathbb{R}}							
\newcommand{\Wholes}{\mathbb{Z}}							
\newcommand{\abs}[1]{\left\vert#1\right\vert}			
\newcommand{\sref}[1]{(\ref{#1})}                       

\newtheorem{thm}{Theorem}[section]
\newtheorem{cor}[thm]{Corollary}

\numberwithin{equation}{section}

\usepackage{authblk}
\title{\sc Multichromatic travelling waves for lattice Nagumo equations}
\author[1]{Hermen Jan Hupkes \thanks{\tt hhupkes@math.leidenuniv.nl}}
\author[1]{Leonardo Morelli \thanks{\tt leonardo.morelli@gmail.com}}
\author[2]{Petr Stehl\'{\i}k\thanks{corresponding author, \tt pstehlik@kma.zcu.cz}}
\author[2]{Vladimir \v{S}v\'{\i}gler \thanks{\tt sviglerv@kma.zcu.cz}}
\affil[1]{\small Mathematisch Instituut, Universiteit Leiden, P.O. Box 9512, 2300 RA Leiden, The Netherlands}
\affil[2]{\small Department of Mathematics and NTIS, Faculty of Applied Sciences, University of West Bohemia,\authorcr Univerzitn\'\i~8, 306 14 Plze\v{n}\\ Czech Republic}

\usepackage{xcolor}

\begin{document}

\maketitle

\begin{abstract}
We discuss multichromatic front solutions
to the bistable Nagumo lattice differential equation.
Such fronts connect the stable spatially homogeneous equilibria with spatially heterogeneous $n$-periodic equilibria
and hence are not monotonic like the standard monochromatic fronts. In contrast to the bichromatic case,
our results show that these multichromatic
fronts can disappear and reappear as the diffusion coefficient
is increased. In addition, these multichromatic
waves can travel in parameter regimes
where the monochromatic fronts are also free to travel.
This leads to intricate collision processes where an
incoming multichromatic wave can reverse its direction
and turn into a monochromatic wave.




\end{abstract}

\section{Introduction}

In this paper we are interested in the Nagumo 
lattice differential equation (LDE)
\begin{equation}
\label{eq:int:nag:LDE}
\dot{u}_j(t)=d \big[ u_{j-1}(t) -2u_j(t) + u_{j+1}(t)  \big] 
+ u(1-u)(u-a) ,
\end{equation}
posed on the one-dimensional lattice $j \in \Wholes$,
for small values of the diffusion-coefficient $d > 0$.
This so-called anti-continuum regime features spatially periodic equilibria
for \sref{eq:int:nag:LDE} that can serve as buffer zones between regions
of space where the homogeneous stable equilibria $u \equiv 0$ and $u \equiv 1$ dominate the dynamics. Our goal here is to continue the program
initiated in \cite{HJHBICHROM} where two-periodic patterns and their connection to bichromatic waves were 
rigorously analyzed. 

In particular, we build a classification framework that also allows larger periods to be considered. 
We also construct
so-called multichromatic waves, which connect 
heterogeneous $n$-periodic rest-states
to each other or their homogeneous
counterparts $u \equiv 0$ and $u \equiv 1$.
%
We numerically analyze these multichromatic waves for $n \in \{3, 4 \}$ 
and show that they 
exhibit richer behaviour than 
the bichromatic versions. 
Indeed, these multichromatic waves can 
disappear and reappear as the diffusion $d$ is increased. 
In addition, open sets of parameters $(a,d)$ exist
where monochromatic and multichromatic waves
can travel simultaneously. This allows us to explore
various new types of wave collisions.


\paragraph{Reaction-diffusion systems}
The LDE \sref{eq:int:nag:LDE} can be seen as the spatial discretization
of the Nagumo PDE
\begin{equation}
    \label{eq:int:nag:pde}
    u_t = u_{xx} + u (1 - u)(u -a)
\end{equation}
onto a uniform grid with node-spacing $h = d^{-1/2}$. This scalar reaction-diffusion PDE can serve as a highly simplified model
to describe the interaction between two species or states 
(described by $u =0$
and $u = 1$) that compete for dominance in a spatial domain \cite{Aronson1975nonlinear}. It admits a comparison principle
and can be equipped with a variational structure \cite{Gallay2007variational},
but it also has a rich global attractor. As such, it has served
as a prototype system to investigate many of the key concepts in the field
of pattern formation, such as 
spreading speeds for compact disturbances \cite{WEINBERGER1978},
the existence and 
stability of travelling waves \cite{Fife1977,SATTINGER1977}
and other non-trivial entire solutions
\cite{morita2006entire,Yagisita2003backward}.

The semi-discrete version \sref{eq:int:nag:LDE} has served as a
playground to investigate the impact of the 
transition from a spatially continuous to a spatially
discrete domain. From a mathematical point of view,
interesting questions and complications arise due to the broken translational 
invariance \cite{VL30}. From the practical point of view, it is highly desirable
to be able to incorporate the natural spatial discreteness
present in many physical systems such as
myelinated nerve fibres \cite{RANVIER1878}, 
meta-materials \cite{CAHNNOV1994,CAHNVLECK1999,VAIN2009}
and crystals \cite{Celli1970motion,Dmitriev2000domain}.

\paragraph{Monochromatic waves}
Substitution of the travelling wave Ansatz $u(x,t) =\Phi(x - ct)$
into the PDE \sref{eq:int:nag:pde} yields the travelling
wave ODE
\begin{equation}
\label{eq:int:wave:ode}
-c \Phi' = \Phi'' +  \Phi(1 - \Phi)(\Phi - a).
\end{equation}
On the other hand, 
substitution of the discrete analog $u_j(t) = \Phi(j - ct)$
into the LDE \sref{eq:int:nag:LDE} yields
the monochromatic wave equation
\begin{equation}
\label{eq:int:wave:mfde}
    -c \Phi'(\xi) = 
     d \big[ \Phi(\xi - 1) - 2 \Phi(\xi) +  \Phi(\xi + 1) \big]
       + \Phi(\xi)\big(1 - \Phi(\xi)\big)\big( \Phi(\xi) - a\big) ,
\end{equation}
which is a functional differential equation of mixed type (MFDE).
We use the term monochromatic here to refer to the fact
that each spatial index $j$ follows the same waveprofile $\Phi$.
We are specially interested in waves that connect the two
stable equilibria $u= 0$ and $u = 1$. In particular, we impose
the boundary conditions
\begin{equation}
  \label{eq:int:bnd:cnds}
    \Phi(-\infty) = 0,
    \qquad
    \Phi(+\infty) = 1.
\end{equation}

The ODE \sref{eq:int:wave:ode} with \sref{eq:int:bnd:cnds}
can be analyzed by phase-plane analysis \cite{Fife1977}
(and even solved explicitly) to yield the existence
of solutions that increase monotonically
and have $\mathrm{sign}(c) = \mathrm{sign}(a - \frac{1}{2})$.
These waves have a large basin of attraction \cite{Fife1977}
and can be used as building blocks to construct and analyze more complicated
solutions \cite{WEINBERGER1978}.

More advanced techniques are required to analyze \sref{eq:int:wave:mfde},
but again it is possible to show that non-decreasing solutions
exist \cite{MPB}. However, it is now a very delicate question
to determine whether the uniquely determined wavespeed $c = c_{\mathrm{mc}}(a,d)$
satisfies $c_{\mathrm{mc}}(a,d) = 0$ or $c_{\mathrm{mc}}(a,d) \neq 0$. 
Indeed, the broken translational invariance causes
an energy-barrier that must be overcome before waves are able to travel.
As such, there is an open region in the $(a,d)$-plane for which
$c_{\mathrm{mc}}(a,d) = 0$ holds; see Fig. \ref{quad.fig:speed_thres_complete}.
This pinning phenomenon is generic 
\cite{BatesDiscConv,CMPVV, EVV,EVV2005AppMath,HOFFMPcrys,HJHVL2005,VL28,MPCP}
but not omnipresent
\cite{ELM2006,HJH2011} in discrete systems and has received considerable attention.

\paragraph{Bichromatic waves}
The discrete second derivative allows \sref{eq:int:nag:LDE}
to have a much larger class of equilibrium solutions
than the PDE \sref{eq:int:nag:pde}. For example,
two-periodic equilibria of the form
\begin{equation}
    u_j = \left\{ \begin{array}{lcl} v & & j \hbox{ is even}, \\
                   w & & j \hbox{ is odd}. 
                   \end{array} \right.
\end{equation}
can be found by solving the two-component system $G(v,w; a, d) = 0$
given by
\begin{equation}
G (v,w;a,d) = 
\begin{pmatrix}
2d (w -v) + v(1 - v)(v - a) \\
2d (v - w) + w(1 -w)(w - a) \\
\end{pmatrix}.
\end{equation}
The variable $w$ can be readily eliminated, leading to 
a ninth-order polynomial equation for the remaining component $v$.
For $0 < d \ll 1$ this polynomial has nine roots, leading to 
two stable and four unstable two-periodic equilibria for \sref{eq:int:nag:LDE}
besides the three spatially homogeneous equilibria $\{0, a , 1\}$.

In \cite{HJHBICHROM} we performed a full rigorous analysis of this system,
which shows that the number of these two-periodic equilibria
decreases as $d > 0$ is increased. In particular,  
there exist two functions $0 < d_s(a) < d_u(a)$ defined for
$a \in (0,1)$ so that the two stable patterns $(v_{\mathrm{bc}}, w_{\mathrm{bc}})$
and $(w_{\mathrm{bc}} , v_{\mathrm{bc}})$ collide 
with two unstable patterns and disappear as $d$ crosses $d_s(a)$.
The remaining two unstable patterns subsequently collide with $(a,a)$
as $d$ crosses $d_u(a)$, leaving only the three spatially
homogeneous equilibria. We emphasize that all these two-periodic
equilibria only exist in the region where monochromatic waves are pinned, i.e. $c_{\mathrm{mc}}(a,d ) =0$.

Based on general results in \cite{CHENGUOWU2008} we showed that
for $0 < d< d_s(a)$ the system
\sref{eq:int:nag:LDE} admits two
types of bichromatic waves
\begin{equation}
  u_j(t) = 
  \left\{ \begin{array}{lcl} \Phi_e(j -ct) & & j \hbox{ is even}, \\
                   \Phi_o(j-ct) & & j \hbox{ is odd}. 
                   \end{array} \right.
\end{equation}
The first class satisfies  the lower limits
\begin{equation}
    \lim_{\xi \to - \infty} \big(\Phi_e(\xi), \Phi_o(\xi) \big) = (0,0),
    \qquad
    \lim_{\xi \to + \infty} \big(\Phi_e(\xi), \Phi_o(\xi) \big)
    = (v_{\mathrm{bc}},w_{\mathrm{bc}}),
\end{equation}
and has wavespeed $c_{0 \rightarrow \mathrm{bc}} \ge 0$,
while the second class satisfies the
upper limits
\begin{equation}
    \lim_{\xi \to - \infty} \big(\Phi_e(\xi), \Phi_o(\xi) \big) = (v_{\mathrm{bc}},w_{\mathrm{bc}}),
    \qquad
    \lim_{\xi \to + \infty} \big(\Phi_e(\xi), \Phi_o(\xi) \big) 
    = (1,1),
\end{equation}
and has $c_{\mathrm{bc}\to 1} \le 0$. In \cite{HJHBICHROM} we showed
that there exist two thresholds 
\begin{equation}
\label{eq:int:threshold:ex}
    0 < d_{0 \to \mathrm{bc}}(a) \le d_s(a),
    \qquad
    0 < d_{\mathrm{bc} \to 1}(a) \le d_s(a),
\end{equation}
so that in fact $c_{0 \to \mathrm{bc}} > 0$
respectively $c_{\mathrm{bc} \to 1} < 0$ holds as
$d$ is increased above these thresholds.
In addition, for all $a \in (0,1)$ one or both of the
inequalities in \sref{eq:int:threshold:ex} is strict,
indicating the presence
of one or more \textit{travelling} bichromatic waves
for $d$ sufficiently close to $d_s(a)$.

Numerical results indicate that these two types of bichromatic waves
can be glued together via an intermediate buffer zone
that displays the two-periodic pattern $(v_{\mathrm{bc}}, w_{\mathrm{bc}})$. This buffer zone
is consumed as the waves move towards each other and eventually collide
to form a trapped monochromatic wave;
see Figure \ref{tric.fig:LDE_C_D} for the trichromatic analogue.

Bichromatic waves have also been found in several other spatially discrete settings. The results in \cite{BRUC2011,Vainchtein2015propagation}
apply to an anti-diffusion version of \sref{eq:int:nag:LDE} where $d< 0$.
This can reformulated as a two-component problem with positive alternating
diffusion coefficients, 
allowing the general results in \cite{CHENGUOWU2008} to be applied.
Several versions of the two-periodic FPU problem are considered in \cite{Faver2017nanopteron,Faver2018exact, Hoffman2017nanopteron}.
Using a different palette of techniques, the authors obtain so-called 
nanopteron solutions, which can have small high-frequency oscillations
in their tails. Finally, the two-periodic FitzHugh-Nagumo problem
was considered in \cite{schouten2018nonlinear} using a modified
spectral-convergence argument.

\paragraph{Multichromatic waves}
The main purpose of the present paper is to illustrate
the novel behaviour that arises for \sref{eq:int:nag:LDE} 
when considering wave connections to/from
stable $n$-periodic patterns with $n \ge 3$.
Our two main conclusions are that the monotonicity
properties described above are no longer valid
and that travelling multichromatic waves
can co-exist with travelling monochromatic waves.
In particular, travelling multichromatic waves
can appear, disappear and reappear as $d > 0$ is increased
and can collide with other multichromatic waves to
form travelling monochromatic waves.

Since the degree of the polynomial that governs
the $n$-periodic equilibria is given by $3^n$,
it is essential to develop an appropriate classification
system to keep track of all the roots and their
ordering properties. We develop such a system in this paper,
using words from the set $\{\mathfrak{0} , \mathfrak{a}, \mathfrak{1}\}^n$
to track roots that bifurcate off the corresponding sequence
of zeroes of the cubic at $d = 0$. The lack of monotonicity with respect
to $d$ leads to complications and forces us to allow both parameters
$(a,d)$ to vary when tracking roots of a related algebraic problem, unlike in \cite{HJHBICHROM}.
Although the general theory in \cite{CHENGUOWU2008} also applies
to our setting, it is still a challenge to check the conditions
in a systematic fashion.

\paragraph{Wave collisions}
Understanding the interaction between waves
is an important topic that is attracting considerable
attention, primarily in the spatially continuous setting at present.
The so-called weak interaction regime where the waves are far apart is relatively well-understood; see e.g. 
the exit manifold developed by Wright and Hoffman \cite{HW11}
for the discrete setting
and the numerous studies on
renormalization techniques for the continuous setting
\cite{bellsky2013adiabatic,doelman2007nonlinear,van2010front}.

However, at present there is no general theory to understand strong 
interactions, where the core of the waves approach each other
and deform significantly.
Early numerical results by Nishiura and coworkers
\cite{nishiura2003scattering} for the Gray-Scott
and a three-component FitzHugh-Nagumo system
suggest that the fate of colliding waves (annihilation, 
combination or scattering) is related
to the properties of a special class of unstable solutions
called separators. Even the internal dynamics
of a single pulse under the influence of essential spectrum
(a proxy for the advance of a second wave)
can be highly complicated, see e.g.
\cite{chirilus2015butterfly} for the (partial)
unfolding of a butterfly catastrophe.

Naturally, more information can be obtained in the presence of a comparison principle. Indeed, for the PDE \sref{eq:int:nag:pde}
one can show that monostable waves can merge
to form a bistable wave \cite{morita2006entire}
and that counterpropagating waves can annihilate
\cite{morita2009entire,Yagisita2003backward}. If one modifies
the nonlinearity to allow more zero-crossings,
one can stack waves that connect a chain of equilibria 
to form so-called propagating terraces \cite{polacik2016propagating}.

We emphasize that the collisions described in this paper
are far richer than those described above for \sref{eq:int:nag:pde}.
This is a direct consequence of the delicate structure of the
set of equilibria for \sref{eq:int:nag:LDE}. By exploiting the comparison
principle, our hope is that this system can serve as a playground
for generating and understanding complicated collision processes.

\paragraph{Organization}
In {\S}\ref{sec:mcr} we discuss the algebraic problem
that $n$-periodic equilibria must satisfy, develop a naming scheme
for its roots and formulate a result concerning the existence
of travelling waves. In {\S}\ref{sec:tri}
we discuss trichromatic waves and focus on the fact that for certain values of $a$ three-periodic stable equilibria can disappear and reappear as the diffusion parameter $d$ is increased. We move
on to quadrichromatic waves in {\S}\ref{sec:quad},
highlighting the fact 
that quadrichromatic and monochromatic waves
can travel simultaneously in certain parameter regions. This  allows us to study various 
types of wave collisions.
Finally, in {\S}\ref{sec:pmr} we prove our main result Theorem \ref{thm:mcr:waves}, which establishes the existence of multichromatic waves.

\paragraph{Acknowledgments}
HJH acknowledges support from the Netherlands Organization for Scientific Research (NWO)
(grant 639.032.612). LM acknowledges support from the Netherlands Organization for
Scientific Research (NWO) (grant 613.001.304). PS acknowledges the support of the project LO1506 of the Czech Ministry of Education, Youth and Sports under the program NPU I. The authors are grateful to Anton\'\i n Slav\'\i k for his comments.

\section{Multichromatic Root Naming and Ordering}
\label{sec:mcr}

The main focus of this paper is the Nagumo lattice differential equation (LDE) \begin{equation}
\label{eq:mcr:nagumo:lde}
\dot{u}_j(t)=d \big[ u_{j-1}(t) -2u_j(t) + u_{j+1}(t)  \big] + g\big(u_j(t); a\big), \qquad j \in \Wholes,
\end{equation}
in which the parameters $(a,d)$ are taken from the half-strip
\begin{equation}
  \label{eq:mcr:def:h}
    \mathcal{H} = [0,1] \times [0, \infty)
\end{equation}
and the nonlinearity is given by the cubic
\begin{equation}
g\big(u; a\big) = u(1-u)(u-a) .
\end{equation}
Our results focus on
$n$-periodic stationary solutions
to \sref{eq:mcr:nagumo:lde}
and the waves that connect them.

In {\S}\ref{sec:mcr:naming} we develop a naming
system that allows us to partially classify these
stationary solutions in an intuitive fashion.
We proceed in {\S}\ref{sec:mcr:waves} by formulating
a result for the existence of waves that uses
our naming system to decide which equilibria can be connected.
Finally, equivalence classes for these waves 
are introduced in {\S}\ref{sec:mcr:equiv} by exploiting the symmetries present in \sref{eq:mcr:nagumo:lde}.

\subsection{Equilibrium types}
\label{sec:mcr:naming}

We will write $n$-periodic equilibria for the LDE
\sref{eq:mcr:nagumo:lde} in the form
\begin{equation}
u_i = \gu_{\mod(i,n)}
\end{equation}
for some vector $\gu \in \Real^n$, where we let
the modulo operator take values in 
\begin{equation}
    \mod(i,n) \in \{1, \ldots , n \} .
\end{equation}
We remark that $\gu$ can be interpreted as a solution of the 
Nagumo equation posed on a cyclic graph of length $n$; see \cite{Stehlik2017}.
Taking $n \ge 3$
and introducing the nonlinear mapping
\begin{equation}
\label{eq:mcr:def:G}
G(\gu;a,d) := \begin{pmatrix}
d (\gu_n-2\gu_1+\gu_2) + g\big(\gu_1; a\big) \\
d (\gu_1-2\gu_2+\gu_3) + g\big(\gu_2; a\big) \\
\vdots \\
d (\gu_{n-1}-2\gu_n+\gu_1) + g\big(\gu_n; a\big)
\end{pmatrix}
\in \Real^n,
\end{equation}
we see that any such equilibrium must satisfy $G(\gu;a,d) = 0$.


For any $a \in (0,1)$ and $\gu \in \{0, a, 1\}^n$, it is easy to see that
$G(\gu; a, 0) = 0$ and to confirm that the diagonal matrix
\begin{equation}
\label{eq:mcr:diag:matrix:d1g}
    D_1 G (\gu  ; a, 0) = \mathrm{diag}\Big( g'( \gu_1; a \big), \ldots , g'\big(\gu_n; a \big) \Big)
\end{equation}
has non-zero entries.
In particular, the implicit function theorem implies that
each of these $3^n$ roots is part of a smooth one-parameter family of roots 
that exists whenever $\abs{d}$ is small. In fact, one can track
the location of each of these roots as $d$ is increased, 
up until the point where the root in question disappears by colliding with another root.
This procedure forms the heart of the naming scheme that we develop here,
which will allow us to refer to different types of roots in an efficient manner.

In particular, we set out to label solutions of the equation $G( \cdot \,; a, d) = 0$
with words $\gw$ taken from the set $\{\mathfrak{0} , \mathfrak{a}, \mathfrak{1} \}^n$. 
We emphasize that we are using the fixed symbol $\mathfrak{a}$ as placeholder for 
the parameter $a \in (0,1)$,
which is allowed to vary. 
Indeed, for any 
$\gw \in \{\mathfrak{0} , \mathfrak{a}, \mathfrak{1} \}^n$ we introduce the 
vector $\gw_{\mid a} \in \Real^n$ by writing
\begin{equation}
\big(\gw_{\mid a}\big)_i = \left\{ \begin{array}{lcl} 
                             0 & & \hbox{if } \gw_i = \mathfrak{0}, \\
                             a & & \hbox{if } \gw_i = \mathfrak{a},  \\
                             1 & & \hbox{if } \gw_i = \mathfrak{1} .
                         \end{array}
                        \right.
\end{equation}

\begin{definition}\label{d:mcr:eq:type}
Consider a word $\gw \in \{\mathfrak{0} , \mathfrak{a}, \mathfrak{1} \}^n$
together with a triplet
\begin{equation}
(\gu, a,d) \in [0,1]^n \times (0, 1) \times [0, \infty) .
\end{equation}
Then we say that $\gu$ is an equilibrium
of type $\gw$ if there exists a $C^1$-smooth curve
\begin{equation}
\label{eq:gamma:curve}
[0,1] \ni t \mapsto \big( \gv(t),  \alpha(t), \delta(t)  \big) \in [0,1]^n \times (0,1) \times [0, \infty)
\end{equation}
so that we have
\begin{equation}
\label{eq:mc:path:reqs:basic}
\begin{array}{lcl}
  (\gv, \alpha,\delta)(0) & = & (\gw_{\mid a}, a , 0),
\\[0.2cm]
  (\gv, \alpha,\delta)(1) & = & (\gu, a , d),
\end{array}
\end{equation}
together with
\begin{equation}
\label{eq:mc:path:reqs}
    G\big( \gv(t) ; \alpha(t), \delta(t) \big) = 0,
    \qquad
    \det D_1 G\big( \gv(t); \alpha(t), \delta(t) \big) \neq 0
\end{equation}
for all $0 \le t \le 1$.
\end{definition}

We note that substituting $t =1$ in \sref{eq:mc:path:reqs} shows that indeed $G(\gu ; a,d ) = 0$,
justifying the terminology of an equilibrium. In addition, the second requirement in \sref{eq:mc:path:reqs}
allows us to apply the implicit function theorem to conclude that 
$G( \cdot \, ; \tilde{a}, \tilde{d}) = 0$
also has equilibria of type $\gw$ for all pairs $(\tilde{a},\tilde{d})$ sufficiently close to $(a,d)$.
In particular, 
these observations allow us to introduce the pathwise connected set
\begin{equation}
\Omega_{\gw} = \{ ( a, d) \in \mathcal{H} : \hbox{the system } G
  (\cdot \, ;a, d) = 0 \hbox{ admits an equilibrium of type } \gw \},
\end{equation}
which is open in the half-strip $\mathcal{H} =  [0,1] \times [0, \infty)$.

We now impose the following conditions on the structure of these sets $\Omega_{\gw}$.
The second of these basically states
that the interior of the curve \sref{eq:gamma:curve} can be perturbed freely within $\Omega_{\gw}$ without changing
the value of the equilibrium $\gu$. 

\begin{itemize}
  \item[${\rm (H\Omega 1)}$]{
    For any two words $\gw_A,\gw_B \in \{\mathfrak{0} , \mathfrak{a}, \mathfrak{1} \}^n$ the intersection
    $\Omega_{\gw_A} \cap \Omega_{\gw_B}$ is connected.
  }
  \item[${\rm (H\Omega 2)}$]{
    Consider any word $\gw \in \{\mathfrak{0} , \mathfrak{a}, \mathfrak{1} \}^n$ and any $(a,d) \in \Omega_\gw$,
    together with a pair of curves 
    \begin{equation}
        [0,1] \ni t \mapsto (\alpha_A, \delta_A)(t) \in \Omega_{\gw},
        \qquad
        [0,1 ] \ni t \mapsto (\alpha_B, \delta_B)(t) \in \Omega_{\gw}
    \end{equation}
    that have
    \begin{equation}
        (\alpha_A, \delta_A)(0) = (\alpha_B , \delta_B)(0) = (a, 0),
        \qquad
        (\alpha_A, \delta_A)(1) = (\alpha_B , \delta_B)(1) = (a, d).
    \end{equation}
    Then there exist unique functions
    \begin{equation}
        [0,1] \ni t \mapsto (\gv_A, \gv_B)(t) \in \Real^n \times \Real^n
    \end{equation}
    so that the triplets $(\gv_A , \alpha_A , \delta_A)$
    and $(\gv_B, \alpha_B, \delta_B)$
    both satisfy \sref{eq:mc:path:reqs} for all $0 \le t \le 1$,
    together with the identities
    \begin{equation}
          \gv_A(0) =  \gv_B(0) = \gw_{\mid a},
          \qquad
          \gv_A(1) = \gv_B(1).
    \end{equation}
  }
\end{itemize}

At first glance, the condition ${\rm (H\Omega 2)}$ appears to be rather cumbersome to verify in practice.
To make this more feasible, it is useful to introduce the set 
\begin{equation}
\label{eq:mcr:def:Gamma}
\Gamma = \{ ( a,d) \in \mathcal{H}: \hbox{ there exists } 
  u \in \Real^n \hbox{ for which } G( u; a , d) = 0
  \hbox{ and } \det D_1 G(u; a, d) = 0
\},
\end{equation}
which by continuity is a closed subset of 
$\mathcal{H} = [0,1] \times [0, \infty)$.
Let us assume that $\mathcal{H} \setminus \Gamma$ consists
of a finite number $N_{\Gamma}$ of components
$\{V_i\}_{i=1}^{N_{\Gamma}}$ that are open in $\mathcal{H}$ and \emph{simply connected}. 
On account of a global implicit function
theorem \cite[Thm. 3]{blot1991global},
there exist non-negative integers $\{m_i\}_{i=1}^{N_{\Gamma}}$
together with smooth functions
\begin{equation}
\gu_{i,j}: V_i \to \Real^n, \qquad \qquad 1 \le i \le N_{\Gamma}, \qquad 1 \le j \le m_i
\end{equation}
so that $G(\gu_{i,j}(a,d);a,d) = 0$ for each $(a,d) \in V_i$.
In addition, $\gu_{i,j_1}(a,d) \neq \gu_{i,j_2}(a,d)$ whenever $j_1 \neq j_2$
and every solution $G(\cdot;a,d) = 0$ with $(a,d) \in V_i$ can be written in this way. 

In this setting, ${\rm (H\Omega 2)}$ can be verified by checking
which of these functions can be connected continuously
through the boundaries of adjacent components. Indeed, ${\rm (H\Omega 2)}$
is satisfied if there is no sequence of connections that starts 
in $u_{i,j_1}$ and ends in $u_{i,j_2}$ for some $1 \le i \le N_\Gamma$ and some pair $1 \le j_1 \neq j_2 \le m_i$.

In any case, writing $V_*$ for the component of $\mathcal{H} \setminus \Gamma$ that contains the horizontal segment $(0,1) \times \{0\}$, 
${\rm (H\Omega 2)}$ can always be achieved if one replaces $\Omega_{\gw}$ by subsets of the form
\begin{equation}
\Omega^*_{\gw} = \{ ( a, d) \in V_* : \hbox{the system } G(\cdot;a, d) = 0 \hbox{ admits an equilibrium of type } \gw \}.
\end{equation}
However, we emphasize that ${\rm (H\Omega 2)}$ appears to be valid without this artificial restriction
for the regions that we have numerically computed in this paper.

\begin{cor}
\label{cor:mcr:unq:props}
Fix an integer $n \ge 2$ and suppose that ${\rm (H\Omega 1)}$ and ${\rm (H\Omega 2)}$ both hold. Then for any word $\gw \in \{\mathfrak{0} , \mathfrak{a}, \mathfrak{1} \}^n$
there is a smooth function $\gu_{\gw}: \Omega_{\gw} \to \Real^n$ so that
$\gu_\gw( a, d)$ is the unique equilibrium of type $\gw$ for the system $G(\cdot \, ; a,d) = 0$
for all $(a,d) \in \Omega_{\gw}$. In addition, whenever $(a,d) \in \Omega_{\gw_A} \cap \Omega_{\gw_B}$
for two distinct words $\gw_A, \gw_B \in \{\mathfrak{0} , \mathfrak{a}, \mathfrak{1} \}^n$ we have 
\begin{equation}
\label{eq:mcr:cor:u:w:uneq}
  \gu_{\gw_A}(a,d) \neq \gu_{\gw_B}(a,d).
\end{equation}
\end{cor}
\begin{proof}
The first statement is a consequence of ${\rm (H\Omega 2)}$ and the implicit function theorem. For the second statement,
let us argue by contradiction and assume that both vectors in \sref{eq:mcr:cor:u:w:uneq} are equal to $\gu \in \Real^n$.
Condition ${\rm (H\Omega 1)}$ allows us to pick a path from $(a,0)$ to $(a,d)$ that lies entirely within $\Omega_{\gw_A} \cap \Omega_{\gw_B}$.
Applying ${\rm (H\Omega 2)}$ shows that $\gu$ can be continued as an equilibrium along this path back to both $(\gw_A)_{\mid a}$
and $(\gw_B)_{\mid a}$. However, the implicit function theorem implies that this continuation should be unique.
\end{proof}

We emphasize that our classification scheme only track roots until the first time the associated Jacobian becomes singular
and the implicit function theorem can no longer be applied.
We will encounter several different types of behaviour at such points. It is possible for two roots
to collide and become complex and sometimes even recombine at `later' parameter values. We will also encounter multi-root collisions
where one or more roots survive the collision process. 
In such cases, we often use an ad-hoc naming system, where we label the emerging
branch by reusing or combining  the types of the original colliding branches.

\subsection{Waves}
\label{sec:mcr:waves}

In this section we focus on 
wave solutions to \sref{eq:mcr:nagumo:lde}  that connect the $n$-periodic stationary solutions 
investigated in the previous section. These so-called multichromatic
waves can be written as
\begin{equation} \label{eq:tw:ansatz}
u_j(t) = \Phi_{\mod(j,n)}(j - ct)
\end{equation}
for some wavespeed $c \in \Real$ and $\Real^n$-valued waveprofile
\begin{equation}
\Phi = (\Phi_1, \Phi_2,\ldots,\Phi_{n}): \Real \to \Real^n
\end{equation}
that satisfies the boundary conditions
\begin{equation}
  \label{eq:mr:bnd}
    \lim_{\xi \to - \infty}
      \Phi(\xi) = \gu_{\gw_-},
      \qquad
    \lim_{\xi \to + \infty}
      \Phi(\xi) = \gu_{\gw_+}
\end{equation}
for some pair of words 
$\gw_\pm \in \{\mathfrak{0}, 
\mathfrak{1} \}^n$; see Corollary \ref{cor:mcr:unq:props}.

Substituting this Ansatz into \sref{eq:mcr:nagumo:lde} 
yields the traveling wave functional differential equation
\begin{equation}
\label{eq:mcr:wave:mfde}
\begin{array}{lcl}
- c  \Phi_1'(\xi)
 &=& d \big[ \Phi_{n}(\xi - 1) - 2 \Phi_1(\xi) + \Phi_2(\xi + 1) \big]
    + g\big(\Phi_1(\xi); a \big) ,
\\[0.2cm]
- c  \Phi_2'(\xi)
 &=& d \big[  \Phi_1(\xi - 1) - 2 \Phi_2(\xi) + \Phi_3(\xi + 1) \big]
    + g\big( \Phi_2(\xi); a \big) , \\
&\vdots& \\
- c  \Phi_{n}'(\xi)
 &=& d \big[  \Phi_{n-1}(\xi - 1) - 2 \Phi_{n}(\xi) + \Phi_1(\xi + 1) \big]
    + g\big( \Phi_{n}(\xi); a \big),
\end{array}
\end{equation}
which has positive coefficients on all shifted terms and also has diagonal
nonlinearities. This system hence
fits into the framework developed in \cite{CHENGUOWU2008},
provided that the assumptions pertaining to the boundary
conditions \sref{eq:mr:bnd} can also be validated.

\begin{figure}
\begin{minipage}{\textwidth}
\centering
\begin{minipage}{0.45\textwidth}
\centering
\includegraphics[width=\textwidth]{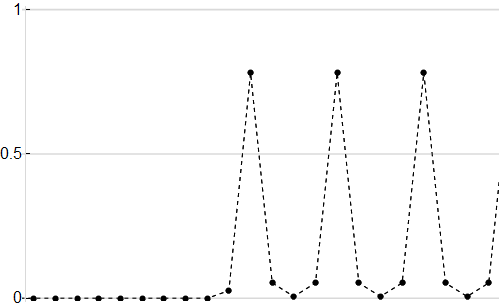} \\
(a) $\gw_-=\mathfrak{0000}, \gw_+=\mathfrak{0001}$.
\end{minipage}\ 
\begin{minipage}{0.45\textwidth}
\centering
\includegraphics[width=\textwidth]{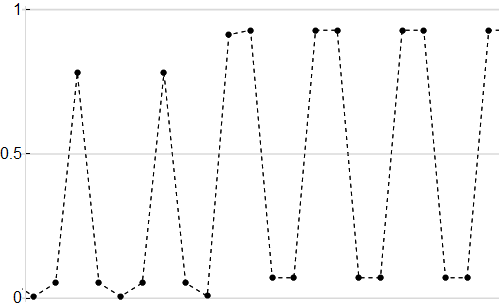} \\
(b) $\gw_-=\mathfrak{0001}, \gw_+ =\mathfrak{0011}$.
\end{minipage} 
\end{minipage}
\caption{Two examples of $4$-chromatic waves as described in Theorem \ref{thm:mcr:waves}.
}\label{f:connections}
\end{figure}

This is
in fact the key question that we address in 
our main result below. 
This result requires the following assumption,
which states that our root tracking scheme captures all 
(marginally) stable\footnote{
  We use the eigenvalues of the matrix $D_1G(\gu; a, d)$
  to characterize the stability of $\gu$. Using the comparison principle
  this can be easily transferred to the
  full LDE \sref{eq:int:nag:LDE}.
}
$n$-periodic equilibria of the LDE.
In particular, the types of these equilibria correspond with words from the \textit{stable}
subset $\{\mathfrak{0},\mathfrak{1}\}^n$.

\begin{itemize}
    \item[(HS)]{
      Recall the definitions \sref{eq:mcr:def:h}
      and \sref{eq:mcr:def:Gamma} and
      suppose that $G(\gu; a, d) = 0$ for some $\gu \in \Real^n$ and $(a,d) \in \mathcal{H} $. 
      Suppose furthermore that
      all eigenvalues $\lambda$ of $D_1 G( \gu;a,d)$ satisfy
      $\lambda \le 0$.
      Then we have\footnote{
        If $(a,d) \in \partial \Omega_{\gw}$
        this identity should be interpreted as a limit.
      } $\gu = \gu_{\gw}(a, d)$ for some 
      $\gw \in \{\mathfrak{0}, \mathfrak{1}\}^n$
      and $(a,d) \in \overline{\Omega}_{\gw}$.
    }
\end{itemize}


\begin{theorem}[{see {\S}\ref{sec:pmr}}]
\label{thm:mcr:waves}
Fix an integer $n \ge 2$ and assume that ${\rm (H\Omega 1)}$, ${\rm (H\Omega 2)}$ and ${\rm (HS)}$ all hold.
Consider two distinct words $\gw_-, \gw_+\in \{\mathfrak{0},\mathfrak{1} \}^n$ with $\gw_- \le \gw_+$
and pick $(a,d) \in \Omega_{\gw_-} \cap \Omega_{\gw_+}$ with $d > 0$.
Suppose furthermore that  one of the following conditions holds.
\begin{itemize}
    \item[(a)]{
      The words $\gw_-$ and $\gw_+$ differ at precisely one location.
    }
    \item[(b)]{
      For each $\gw \in \{\mathfrak{0}, \mathfrak{1}\}^n \setminus \{ \gw_-, \gw_+ \}$ 
      that satisfies $\gw_- \le \gw \le \gw_+$ we have $(a,d) \notin \overline{\Omega}_{\gw}$.
    }
\end{itemize}
Then there exists a unique $c \in \Real$
for which the travelling system \sref{eq:mcr:wave:mfde}
admits a solution $\Phi: \Real \to \Real^n$ 
that satisfies the boundary conditions
\begin{equation}
\label{eq:mcr:wave:bnd:cnds}
    \lim_{\xi \to - \infty} \Phi(\xi) = \gu_{\gw_-}(a,d),
    \qquad
    \lim_{\xi \to + \infty} \Phi( \xi  ) = \gu_{\gw_+}(a,d).
\end{equation}
If $c \neq 0$, then $\Phi$ is unique up to translation and 
each component is strictly increasing.
\end{theorem}

\subsection{Symmetries}
\label{sec:mcr:equiv}

There are a number of useful symmetries present in the 
equilibrium equation $G(\gu;a, d) = 0$
and the travelling wave MFDE \sref{eq:mcr:wave:mfde}. We explore three important
transformations here that significantly reduce the number of cases that need to be considered.

We first note that the identity $g(1 - u; a) = -g(u; 1 - a)$ implies that
\begin{equation}
\label{eq:mcr:inv:symm:G}
G(\mathbf{1} - \gu; a, d) = -G( \gu; 1 - a, d)
\end{equation}
holds for all $\gu \in \Real^n$, with $\mathbf{1} = (1, \ldots, 1)^\top \in \Real^n$.
In order to exploit this, we pick $\gw \in \{\mathfrak{0},\mathfrak{a},\mathfrak{1}\}^n$
and write $\gw_{\mathfrak{0} \leftrightarrow \mathfrak{1}} \in \{ \mathfrak{0}, \mathfrak{a}, \mathfrak{1} \}^n$
for the `inverted' word
\begin{equation}
 \big(\gw_{\mathfrak{0} \leftrightarrow \mathfrak{1}}\big)_i = 
 \left\{ \begin{array}{lcl} 
                             \mathfrak{1} & &\hbox{if } \gw_i = \mathfrak{0}, \\
                             \mathfrak{a} & & \hbox{if } \gw_i = \mathfrak{a},  \\
                             \mathfrak{0} & &\hbox{if } \gw_i = \mathfrak{1}. \\
                         \end{array}
                        \right.
\end{equation}
The identity \sref{eq:mcr:inv:symm:G} directly implies 
the reflection relation
\begin{equation}
\Omega_{\gw_{\mathfrak{0} \leftrightarrow \mathfrak{1}} } = \{ ( a, d) : (1 - a, d) \in \Omega_{\gw} \},
\end{equation}
together with
\begin{equation}
    \gu_{\gw_{\mathfrak{0} \leftrightarrow \mathfrak{1}}} (a, d) = \mathbf{1} - \gu_{\gw}( 1 - a,d) .
\end{equation}

We proceed by introducing the 
coordinate-shifts $\{T_k\}_{k=1}^n: \Real^n \to \Real^n$
and reflection $R: \Real^n \to \Real^n$
that act as
\begin{equation}
\big(T_k \gu\big)_i = \gu_{\mod(i + k , n)},
\qquad \qquad
\big(R \gu \big)_i = \gu_{  \mod( 1-i ,n) } .
\end{equation}
It is easy to verify that the identity $G(\gu; a, d) = 0 $
implies that also
\begin{equation}
    G(T_k \gu; a, d) = G(R \gu ; a, d) =  0 .
\end{equation} 
In addition, when $(c,\Phi)$
is a solution to the travelling wave MFDE \sref{eq:mcr:wave:mfde},
the same holds for the coordinate-shifted pair $(c,T_k \Phi)$ and the reflected version
$(-c, \tilde{\Phi})$ with
\begin{equation}
  \label{eq:mcr:rev:wave}
    \tilde{\Phi}(\xi) = R \Phi(-\xi).
\end{equation}
These identities are all consequence of the invariance
of the Nagumo LDE \sref{eq:mcr:nagumo:lde} under
the transformations $j \mapsto j + k$ and $j \mapsto - j$.

Our choice to only consider waves where $\Phi(-\infty) \le \Phi(+\infty)$
breaks the reflection invariance \sref{eq:mcr:rev:wave},
which allows us to focus solely on the symmetry caused by the 
coordinate shifts $T_k$. In particular,
for any word $\gw \in \{\mathfrak{0},\mathfrak{a},\mathfrak{1}\}^n$,
we write
\begin{equation}
    \Omega_{[\gw]} = \Omega_{\gw} = \Omega_{T_1 \gw} = 
    \ldots = \Omega_{T_{n-1} \gw }.
\end{equation}
In addition, we introduce the shorthand notation
\begin{equation}
    [\gw] 
    = \{ \gu_{\gw} , T_1 \gu_{\gw}, \ldots , T_{n-1} \gu_{\gw} \} 
\end{equation}
to refer to all the roots in the corresponding equivalence class,
which are defined for $(a,d) \in \Omega_{[\gw]}$.

Note that we are hence treating $[\mathfrak{0a1}]$ and $[\mathfrak{01a}]$ as 
separate classes, even though they correspond
with equivalent equilibria of $G(\cdot ; a, d) = 0$.
Interestingly enough, when restricting attention to the alphabet
$\{\mathfrak{0},\mathfrak{1}\}^n$, this distinction 
only plays a role for $n \ge 6$.
For example, $\mathfrak{001011}$ is not shift-related to its reflection $\mathfrak{110100}$,
but all shorter binary sequences are.

\begin{figure}[t]
\begin{center}
\includegraphics[width=1\textwidth]{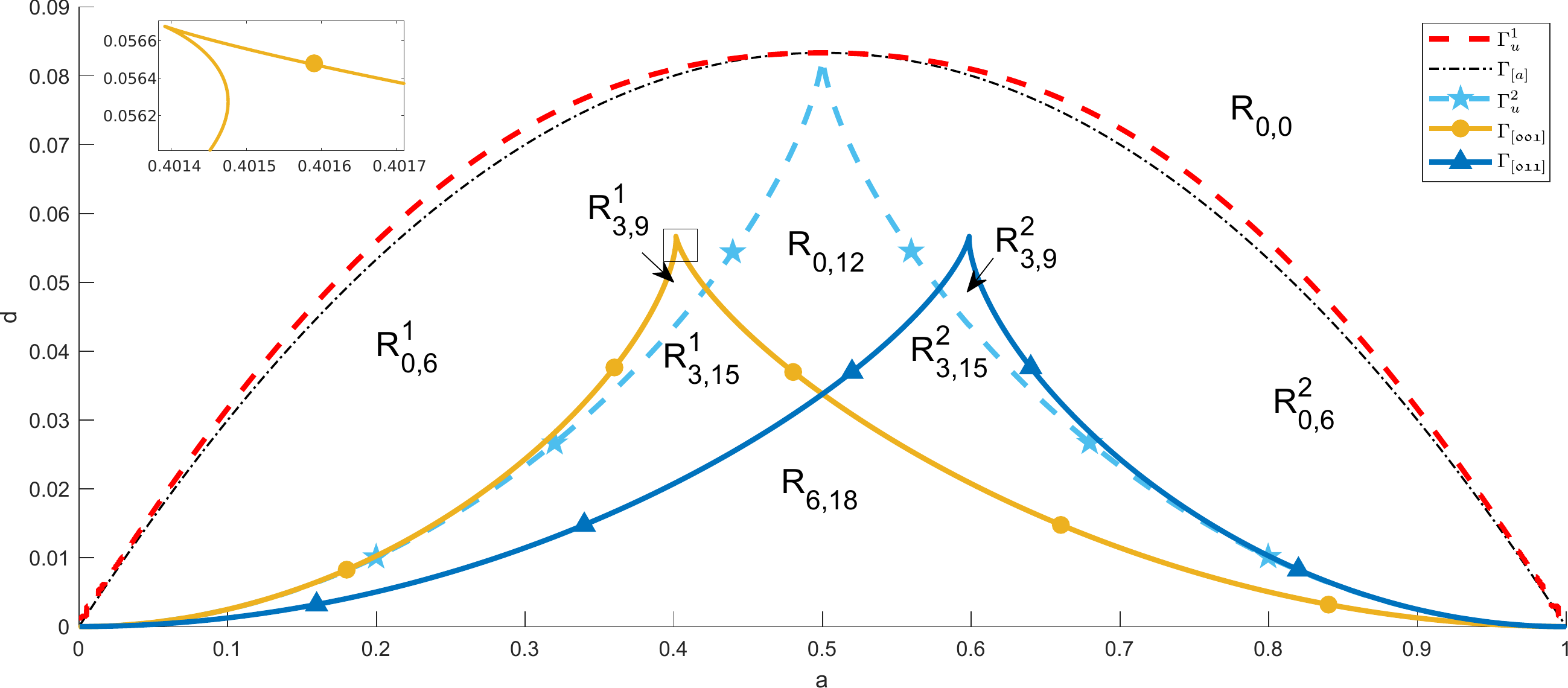}
\caption{Bifurcation thresholds for the number of roots of \eqref{eq:tcr:def:G}. 
The subscripts in the regions $R$ indicate 
the number of stable and unstable roots present in each region,
counting their multiplicities but excluding the homogeneous roots
$[\mathfrak{0}]$,
$[\mathfrak{a}]$ and $[\mathfrak{1}]$.
} 
\label{tric.fig:root_thres}
\end{center}
\end{figure}

Whenever we are referring to an 
equivalence class of roots,
we will use the word that is the smallest in the 
lexicographical\footnote{For ordering purposes
we assume that $\mathfrak{0} < \mathfrak{a} < \mathfrak{1}$ 
holds.
} 
sense (the so-called Lyndon word) as a class representative.
For example,  $\mathfrak{0a11}$ is the Lyndon word for the class
\begin{equation}\label{eq:Lyndon:representative}
[\mathfrak{0a11}] = \{\gu_{\mathfrak{0a11}}, \gu_{\mathfrak{10a1}}, \gu_{\mathfrak{110a}}, \gu_{\mathfrak{a110}} \}.
\end{equation}
We note that shorter words with a length that divides $n$ can also
be interpreted as a word of length $n$ by periodic extension.
For example, we write
\begin{equation}
[\mathfrak{01}] = [\mathfrak{0101}] = \{\gu_{\mathfrak{0101}},\gu_{\mathfrak{1010}}\}.
\end{equation}

For any pair $\gw_\pm \in \{\mathfrak{0}, \mathfrak{1} \}^n$ with $\gw_- \le \gw_+$, 
we will use the 
shorthand notation $\gu_{\gw_-} \to \gu_{\gw_+}$ to refer to any pair $(c,\Phi)$
that satisfies the travelling wave MFDE \sref{eq:mcr:wave:mfde}
together with the boundary conditions \sref{eq:mcr:wave:bnd:cnds}.
The observations above show that $T_k \gu_{\gw_-} \to T_k \gu_{\gw_+}$ then corresponds
with the pair $(c, T_k \Phi)$.

In order to refer to an equivalence class of wave solutions,
we pick $\gw_- \le \gw_+$ and introduce the notation
\begin{equation}
    [\gw_- \to \gw_+] = \{ \gu_{\gw_-} \to \gu_{\gw_+}, T_1 \gu_{\gw_-} \to T_1 \gu_{\gw_+}, 
      \ldots , T_{n-1} \gu_{\gw_-} \to T_{n-1} \gu_{\gw_+} \} .
\end{equation}
Here we always take $\gw_-$ to be the Lyndon representative for 
its equivalence class,
but we emphasize that $\gw_+$  \emph{cannot} always 
be chosen in this way. 
For example, 
\begin{equation}
    [\mathfrak{001} \to \mathfrak{011}]= \{ \gu_{\mathfrak{001}} \to \gu_{\mathfrak{011}}, 
    \gu_{\mathfrak{100}} \to \gu_{\mathfrak{101}}, \gu_{\mathfrak{010}} \to \gu_{\mathfrak{110}} \} 
\end{equation}
refers to a different set of waves than
\begin{equation}
[\mathfrak{001} \to \mathfrak{101}] = \{ \gu_{ \mathfrak{001}} \to \gu_{\mathfrak{101}}, 
\gu_{\mathfrak{100}} \to \gu_{\mathfrak{110}}, 
\gu_{\mathfrak{010}} \to \gu_{\mathfrak{011}} \}.
\end{equation}
Naturally, this distinction disappears when considering
waves that connect to or from one of the monochromatic states 
$[\mathfrak{0}]$ and $\mathfrak{[1]}$.

\section{Trichromatic waves}
\label{sec:tri}

\begin{figure}[t]
\begin{center}
\includegraphics[width=1\textwidth]{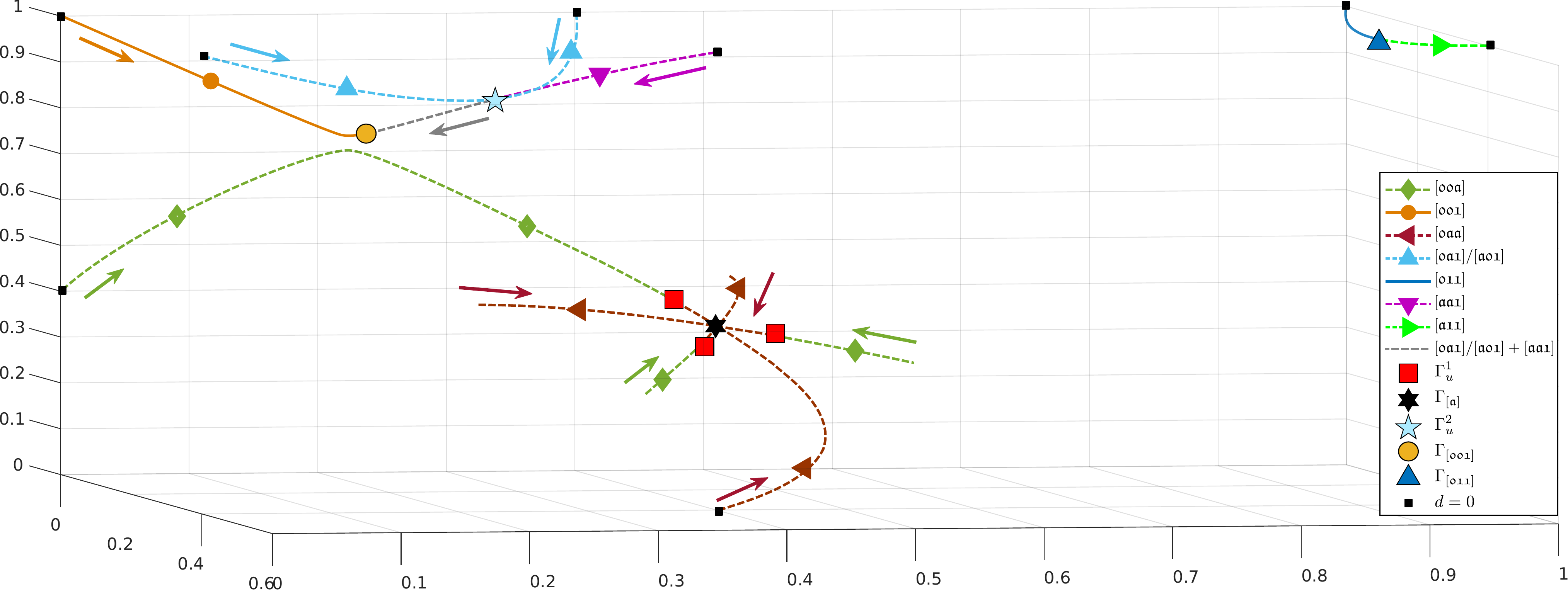}
\caption{Root trajectories as  $d$ is increased for $a=0.401 < a_c$.
The continuous lines represent stable roots branches, while the 
dashed ones are unstable. In principle only one element of each root class is displayed. However, we (locally) plot all three permutations of $\mathfrak{0aa}$
to show how they all collide with 
$(a,a,a)$ at  $\Gamma_{[\mathfrak{a}]}$ and subsequently pass through each other.
Finally, they hit the three permutations of $\mathfrak{00a}$ 
at $\Gamma_u^1$ and disappear. We emphasize that 
$\mathfrak{a01}$ is not a Lyndon word, but we need to use it
here to describe the $\Gamma_u^2$ collision.
}
 \label{tric.fig:root_branch_a401}
\end{center}
\end{figure}

In this section we apply our results to the trichromatic case $n =3$. As in the bichromatic case  considered in \cite{HJHBICHROM}, we observe that  stable $3$-periodic equilibria can only exist inside the parameter region where the monochromatic $[\mathfrak{0} \to \mathfrak{1}]$ wave is pinned. The novel behaviour as compared to the bichromatic case is that there exist intervals
of the parameter $a$ in which the number of (stable) equilibria
can actually increase as $d$ is increased (e.g., for $a=0.40146$, see Figures \ref{tric.fig:root_thres} and \ref{tric.fig:root_branch_a40146}). In addition,
for values of $a$ in these intervals, 
there are two disjoint intervals of parameters $d$ for which \emph{travelling}
trichromatic waves exist that connect the homogeneous
states $[\mathfrak{0}]$ and $[\mathfrak{1}]$ to a stable $3$-periodic
equilibrium. Similarly, for a fixed $d$, the number of (stable) equilibria can increase as $\left| a- \frac{1}{2} \right|$ is increased (e.g., for $d=0.04$, see Figure \ref{tric.fig:root_thres}).

\subsection{Equilibria}
\label{sec:tri:eq}

\begin{figure}[t]
\begin{center}
\includegraphics[width=1\textwidth]{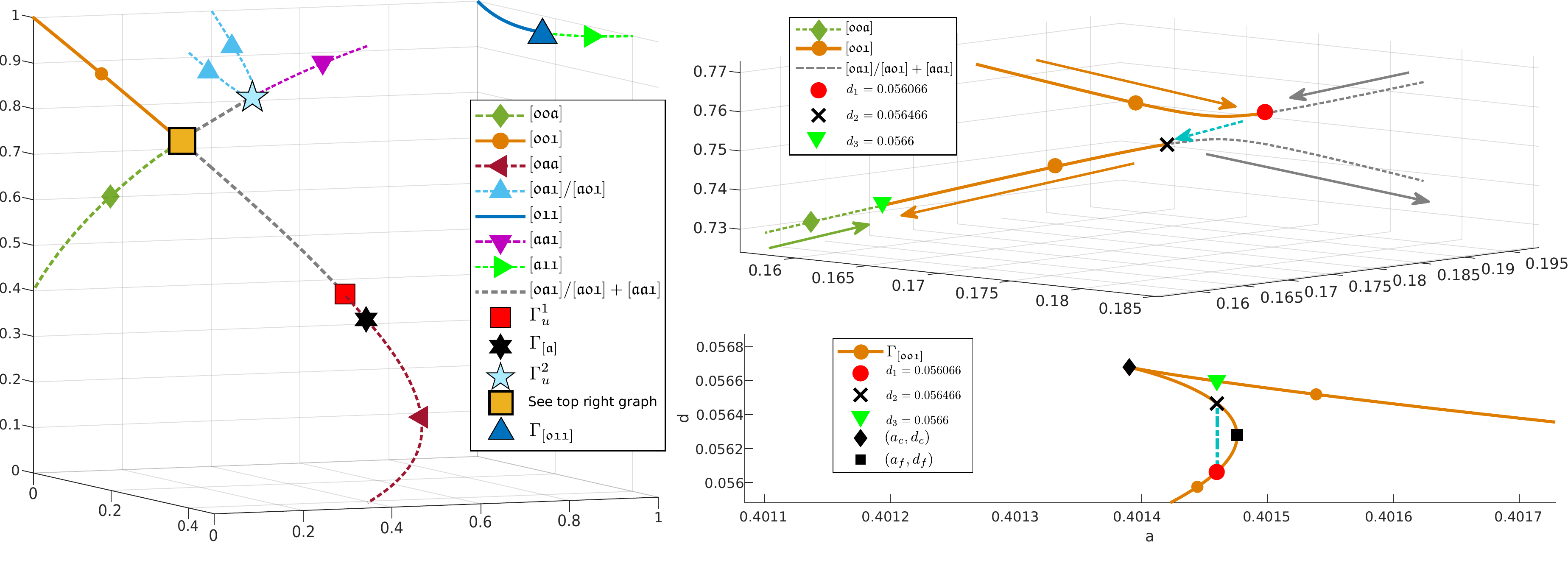}
\caption{Behaviour of the root branches near the cusp
of the yellow $\Gamma_{[\mathfrak{001}]}$ curve from Figure \ref{tric.fig:root_thres},
which is magnified in the bottom right panel.
In particular, we fix $a=0.40146 \in (a_c, a_f)$
and note that  $\Gamma_{[\mathfrak{001}]}$ 
is crossed three times  as $d$ is increased.
The top right panel contains a zoom of the yellow square from the left panel, which is impacted by these crossings.
The blue dashed line between the red circle and black star represents the region where two of the roots are temporarily complex. These plots illustrate the mechanism 
by which the stable branch $[\mathfrak{001}]$ switches
its unstable connecting branch as $a$ is increased through
the cusp and fold points $a_c$ and $a_f$.
} 
\label{tric.fig:root_branch_a40146}
\end{center}
\end{figure}

Trichromatic equilibria $\gu \in \Real^3$ for \sref{eq:mcr:nagumo:lde}
correspond to roots of the nonlinear function
\begin{equation}
\label{eq:tcr:def:G}
G(\gu;a,d) = \begin{pmatrix}
d (\gu_3-2\gu_1+\gu_2) + g\big(\gu_1; a\big) \\
d (\gu_1-2\gu_2+\gu_3) + g\big(\gu_2; a\big) \\
d (\gu_2-2\gu_3+\gu_1) + g\big(\gu_3; a\big) \\
\end{pmatrix}.
\end{equation}
Inspection of this system shows that one component
can be removed if one enforces either $\gu_1 = \gu_2$,
$\gu_2 = \gu_3$ or $\gu_1 = \gu_3$.

In Figure \ref{tric.fig:root_thres} we numerically
computed the critical set $\Gamma$ defined in 
\sref{eq:mcr:def:Gamma} by searching for roots of the augmented system
\begin{equation}
G_*(\gu; a,d) =
  \begin{pmatrix}
    G(\gu; a,d) \\
    \det D_1 G(\gu; a, d)  \\
  \end{pmatrix} .
\end{equation}
The results show that $\Gamma$ can be decomposed into 5 piecewise-smooth curves
that we label as 
$\Gamma_{[\mathfrak{a}]}$, $\Gamma_{[\mathfrak{001}]}$, $\Gamma_{[\mathfrak{011}]}$,
  $\Gamma_u^1$ and  $\Gamma_u^2$. 
The labels of the form 
 $\Gamma_{[\gw]}$ imply that $\Gamma_{[\gw]} \subset \partial \Omega_{[\gw]}$.
%
We emphasize that this naming scheme is ambiguous by its very nature.
Indeed, collisions between roots occur precisely on these curves, which also
intersect each other.

These 5 curves divide the remaining parameter space 
$\mathcal{H} \setminus \Gamma $ into 11 open 
and simply connected components. 
Ignoring the three homogeneous roots\footnote{
In the current trichromatic context, these roots  
are given by $(0,0,0)$, $(a,a,a)$ and $(1,1,1)$.
} $[\mathfrak{0}]$,
$[\mathfrak{a}]$ and $[\mathfrak{1}]$,
the bottom region $R_{6,18}$ contains
$6$ stable roots and $18$ unstable roots.
The stable roots are represented by the equivalence classes
$[\mathfrak{001}]$ and $[\mathfrak{011}]$,
while all the remaining equivalence classes generate
the unstable roots. 

In the discussion below
we indicate how this configuration changes as each of the
critical curves is crossed.
For now, we recall the identity
\sref{eq:mcr:inv:symm:G}, which explains the reflection symmetry through the 
line $a = \frac{1}{2}$ and allows us to focus on the case $a \in (0, \frac{1}{2}]$.

\paragraph{The $\Gamma_{[\mathfrak{011}]}$ threshold}
The first threshold that is encountered when increasing $d$ for $a \in (0, \frac{1}{2})$ is the curve $\Gamma_{[\mathfrak{011}]}$. Here
the stable roots $[\mathfrak{011}]$ collide with the unstable
roots $[\mathfrak{a11}]$, after which both branches disappear.
This collision is visible in Figures \ref{tric.fig:root_branch_a401},
\ref{tric.fig:root_branch_a40146}
and \ref{tric.fig:root_branch_a41_a48}.

In order to find an expansion for this threshold near the corner
$(a,d) = (0,0)$, we exploit the observations above
which allow us to consider equilibria close to $(0,1,1)$
for which the second
and third components are equal.
In particular, we construct solutions
to the problem
\begin{equation}
  G\big( (x, 1 + y, 1 + y) ; a , d\big) = 0
\end{equation}
for which $\abs{x} + \abs{y} + \abs{a}  + \abs{d}$ is small.
The resulting system has a structure that is very similar
to that encountered in \cite{HJHBICHROM}, allowing
us to follow the exact same procedure to unfold
the saddle node bifurcations. In particular,
viewing $\Gamma_{[\mathfrak{011}]}$ locally as the graph of the function 
$d_{[\mathfrak{011}]}$,
we obtain the expansion
\begin{equation}
d_{[\mathfrak{011}]}(a)= \frac{a^2}{8} +\frac{a^4}{64} +O(a^5) ,
\end{equation}
together with
\begin{equation}
 \gu_{011}\big(a ; d_{[\mathfrak{011}]}(a) \big) 
 = \big( \frac{a}{2} ,
 1- \frac{a^2}{8} -\frac{a^3}{16} -\frac{a^4}{64},
 1- \frac{a^2}{8} -\frac{a^3}{16} -\frac{a^4}{64}
 \big) + O( a^5).
\end{equation}

\begin{figure}[t]
\begin{center}
\includegraphics[width=1\textwidth]{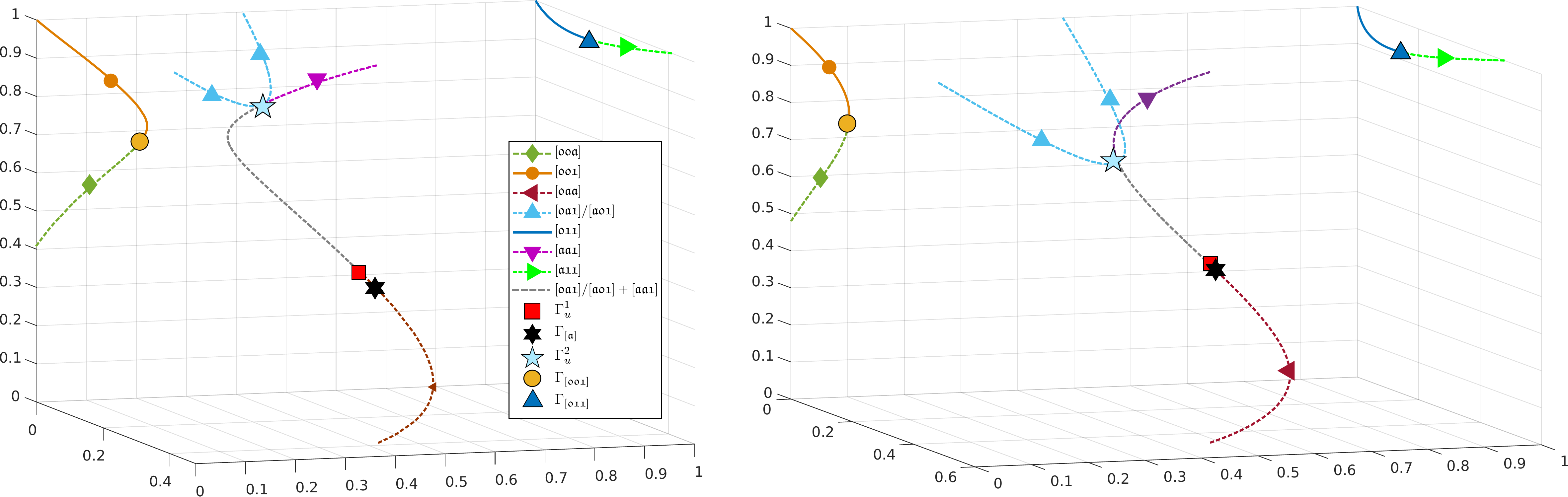}
\caption{Root branches for $a=0.41 > a_f$ (left) and $a=0.48 >a_f$ (right).
The branch $[\mathfrak{001}]$ now collides with $[\mathfrak{00a}]$,
while the unstable branch spawned by the $\Gamma_u^2$ collision 
collides with $[\mathfrak{0aa}]$ at $\Gamma_u^1$.
The collision points $\Gamma_u^1$ and $\Gamma_u^2$ slide through $(a,a,a)$ as 
$a$ is increased beyond 
$a = \frac{1}{2}$.
} 
\label{tric.fig:root_branch_a41_a48}
\end{center}
\end{figure}

\paragraph{The $\Gamma_u^2$ threshold}
This curve features the triple collision of the unstable branches
$[\mathfrak{0a1}]$, $[\mathfrak{a01}]$ and $[\mathfrak{aa1}]$ 
when $a \in (0, \frac{1}{2})$. One unstable branch of roots
emerges from this collision.
This collision is visible in Figures \ref{tric.fig:root_branch_a401},
\ref{tric.fig:root_branch_a40146} and \ref{tric.fig:root_branch_a41_a48}.

\paragraph{The $\Gamma_{[\mathfrak{001}]}$ threshold}
The important feature of this curve is that
$d$ cannot be expressed as a function of $a$. Indeed,
the function
\begin{equation}
    G_{**}(\gu; a, d) = 
      \begin{pmatrix}
         G_*(\gu; a, d) \\
        \det D_{1,3} G_*(\gu; a, d) \\
      \end{pmatrix}
\end{equation}
has the (numerically computed) roots
\begin{equation}
(a_c,d_c)\approx(0.4013889, 0.05668 ),
\qquad  \qquad
(a_f, d_f )\approx(0.401476,0.056275 ),
\end{equation}
which correspond with a cusp respectively fold 
point for the curve $\Gamma_{[\mathfrak{001}]}$;
see Figures \ref{tric.fig:root_thres} and \ref{tric.fig:root_branch_a40146}.

When $a \in (0, a_c)$, the root $[\mathfrak{001}]$ hits the branch spawned by
the $\Gamma_u^2$ collision and disappears; see Figure \ref{tric.fig:root_branch_a401}.
On the other hand, for $a \in (a_f, 1)$ the root
$[\mathfrak{001}]$ hits $[\mathfrak{00a}]$; see Figure \ref{tric.fig:root_branch_a41_a48}. This corresponds
with the scenario described above for 
$\Gamma_{[\mathfrak{011}]}$ after applying
the $\mathfrak{0} \leftrightarrow \mathfrak{1}$ swap.

The intermediate case $a \in (a_c, a_f)$ is illustrated
in Figure \ref{tric.fig:root_branch_a40146}.
In this case the root $[\mathfrak{001}]$ again hits the branch spawned by
the $\Gamma_u^2$ collision and disappears, but this pair reappears
and splits off from each other after $d$ 
crosses the $\Gamma_{[\mathfrak{001}]}$ 
threshold a second time. The stable branch
of this pair collides with $[\mathfrak{00a}]$ and disappears when $d$
crosses $\Gamma_{[\mathfrak{001}]}$ for the third and final time, while
the unstable branch emerges as the survivor of the full triple crossing 
process.

The black star and green triangle in Figure \ref{tric.fig:root_branch_a40146}
overlap when $a = a_c$, in which case the $[\mathfrak{00a}]$ branch
has a triple root at the critical value $d = d_c$.
On the other hand, the black star and red circle in this figure overlap
when $a = a_f$, in which case the branch
$[\mathfrak{001}]$ and the branch spawned by the 
$\Gamma_u^2$ collision can be said
to bounce off each other at $d=d_f$.

In order to find an expansion for this threshold near the corner
$(a,d) = (0,0)$, we now look for solutions to
\begin{equation}
  G\big( (x, x, 1 + y) ; a , d\big) = 0
\end{equation}
for which $\abs{x} + \abs{y} + \abs{a}  + \abs{d}$ is small.
Viewing $\Gamma_{[\mathfrak{001}]}$ locally as the graph of the function 
$d_{\mathfrak{[001]}}$,
we may again use the same procedure as
in \cite{HJHBICHROM} to obtain the expansion
\begin{equation}
d_{\mathfrak{[001]}}(a)= \frac{a^2}{4} +\frac{a^4}{8} +O(a^5) ,
\end{equation}
together with 
\begin{equation}
 \gu_{001}\big(a ; d_{[\mathfrak{001}]}(a) \big) 
 = \big( \frac{a}{2} ,
 \frac{a}{2} ,
 1- \frac{a^2}{8} -\frac{a^3}{16} -\frac{a^4}{64}
 \big) + O( a^5).
\end{equation}

\paragraph{The $\Gamma_{[a]}$ threshold}
This curve is characterized by the relation $\det D_{1}G ([\mathfrak{a}] ; a, d) = 0$,
which can be explicitly solved to yield
$d_{[\mathfrak{a}]}(a) = a(1-a)/3$. As shown in Figure  \ref{tric.fig:root_branch_a401},
the three branches of roots contained in the equivalence class 
$[\mathfrak{0 a a}]$
all pass through $(a,a,a)$ at $d =d_{[\mathfrak{a}]}$ and survive the collision.

\paragraph{The $\Gamma_u^1$ threshold}
On this threshold the unstable branch that survived
the $\Gamma_u^2$ and $\Gamma_{[\mathfrak{001}]}$ collisions hits
the branch $[\mathfrak{0aa}]$ that passed through $(a,a,a)$;
see Figures \ref{tric.fig:root_branch_a401},
\ref{tric.fig:root_branch_a40146},
\ref{tric.fig:root_branch_a41_a48}.
Above this threshold the only remaining equilibria
are the homogeneous states $[\mathfrak{0}]$ 
$[\mathfrak{a}]$ 
and $[\mathfrak{1}]$. 

\paragraph{The critical case $a = \frac{1}{2}$ }
The right panel in Figure  \ref{tric.fig:root_branch_a41_a48} describes
the situation just before $a$ reaches the critical
value $a = \frac{1}{2}$. Upon increasing $a$ through this value,
the collisions at $\Gamma_u^2$ and $\Gamma_u^1$ both cross through the center
$(\frac{1}{2}, \frac{1}{2}, \frac{1}{2})$ at $a = \frac{1}{2}$
and $d = d_{[a]}(\frac{1}{2}) = \frac{1}{12}$,
transitioning to occur on the $[\mathfrak{0aa}]$ branch.

\begin{figure}[t]
\begin{center}
\includegraphics[width=1\textwidth]{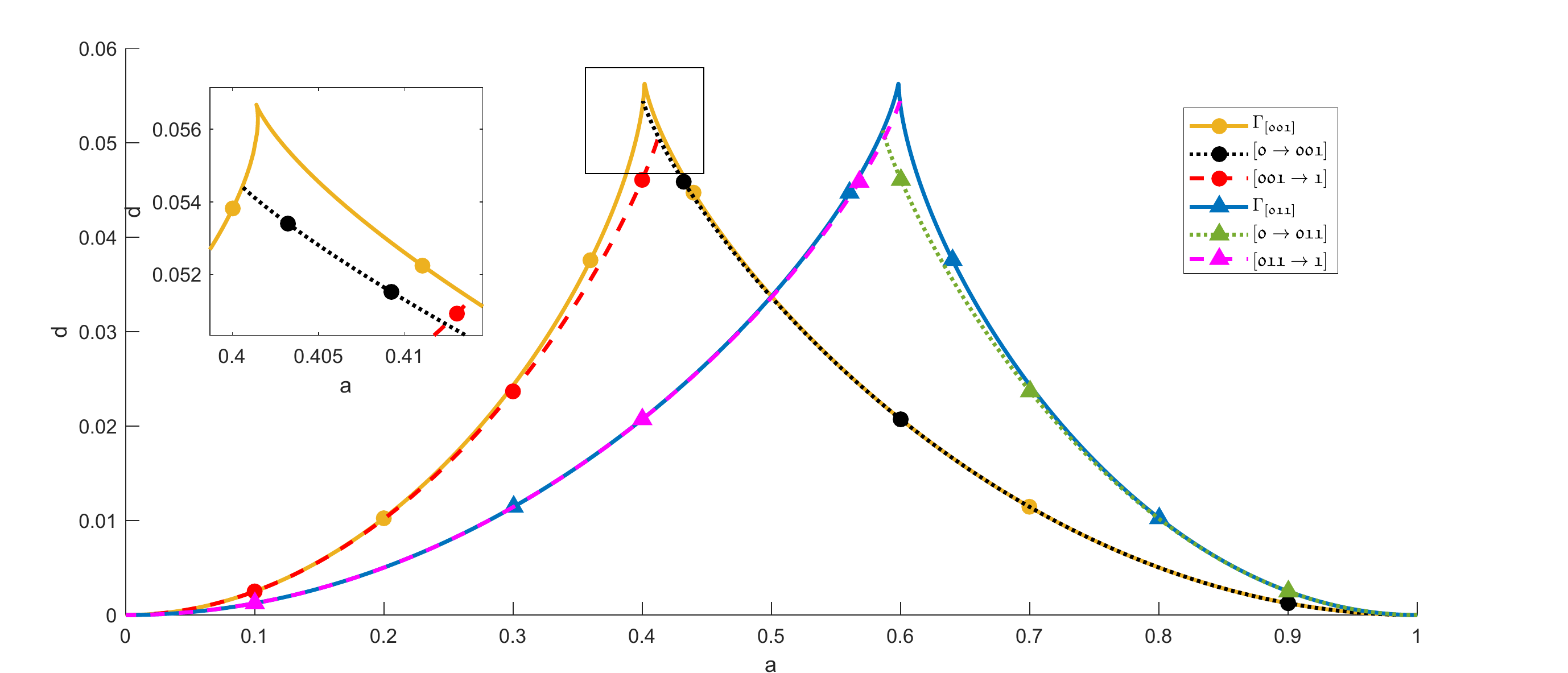}
\caption{Thresholds where standing waves ($c = 0$) transition
to travelling waves ($c \neq 0$)
for the connections to and from
the homogeneous equilibria $[\mathfrak{0}]$ 
and $[\mathfrak{1}]$. 
} 
\label{tric.eq:Speed_thres}
\end{center}
\end{figure}

\subsection{Wave Connections}

Our numerical results strongly suggest that
${\rm (H\Omega 1)}$, ${\rm (H\Omega 2)}$ and ${\rm (HS)}$ are satisfied,
allowing us to apply the results in {\S}\ref{sec:mcr}.
Figure \ref{fig:diag:tri:quad:conn} represents the
equivalence classes of wave connections between neighbouring words.
We note that we are not drawing
an edge from $\mathfrak{101}$ because this is not a Lyndon word.

\begin{figure}[t]
\begin{minipage}{0.5\textwidth}
\begin{center}
\begin{tikzpicture}
\node (a) at (0,2) [circle] {$\mathfrak{0}$};
\node (b) at (2,2) [circle] {$\mathfrak{001}$};
\node (c) at (4,4) [circle] {$\mathfrak{011}$};
\node (d) at (4,1) [circle] {\textcolor{gray}{$\mathfrak{101}$}};
\node (f) at (6,2) [circle] {$\mathfrak{1}$};
\draw (a) edge[style={->}] (b) (b) edge[style={->}] (c) (b) edge[style={->}] (d)  (c) edge[style={->}] (f) ;
\draw (a) edge[bend left={-60},style={->}] (f);
\draw (a) edge[bend left={15},style={->,dashed}] node[above,yshift=0.4cm,xshift=-0.2cm] {$\Omega_{[\mathfrak{011}]} \setminus \overline{\Omega}_{[\mathfrak{001}]}$} (c);
\draw (b) edge[style={->,dashed}] node[above] {$\Omega_{[\mathfrak{001}]} \setminus \overline{\Omega}_{[\mathfrak{011}]} $} (f);
\end{tikzpicture}
\end{center}
\end{minipage}\
\begin{minipage}{0.5\textwidth}
\begin{center}
\begin{tikzpicture}
\node (a) at (0,2) [circle] {$\mathfrak{0}$};
\node (b) at (2,2) [circle] {$\mathfrak{0001}$};
\node (c) at (4,0) [circle] {$\mathfrak{01}$};
\node (d1) at (4,2) [circle] {\textcolor{gray}{$\mathfrak{1001}$}};
\node (d2) at (4,4) [circle] {$\mathfrak{0011}$};
\node (e1) at (6,2) [circle] {$\mathfrak{0111}$};
\node (e2) at (5,5.5) [circle] {\textcolor{gray}{$\mathfrak{1011}$}};
\node (f) at (8,2) [circle] {$\mathfrak{1}$};
\draw (a) edge[->] (b) (b) edge[->] (c) 
   (b) edge[->] (d1) (b) edge[->] (d2) (d2) edge[->] (e1) 
   (d2) edge[->] (e2)  (e1) edge[->] (f) ;
\draw (c) edge[->] (e1) ;
\draw (a) edge[->] (c) (c) edge[->] (f);
\draw (a) edge[bend left={+60},style={->,dashed}] 
  node[above, xshift=0cm,yshift=0.35cm] {$\Omega_{[\mathfrak{0011}]} \setminus \overline{\Omega}_{[\mathfrak{0001}]}$}    (d2);
\draw (a) edge[bend left={-80},style={->,dashed},looseness=1.4] 
node[right,xshift=0,yshift=-0.3cm] 
  {$\Omega_{[\mathfrak{0111}]} \setminus 
    \overline{\Omega}_{[\mathfrak{0001}]} \cup \overline{\Omega}_{[\mathfrak{01}]} \cup \overline{\Omega}_{[\mathfrak{0011}]} $} (e1);
\draw (d2) edge[style={->,dashed}] 
node[above, xshift=-0.2cm,yshift=0.6cm]  {$\Omega_{[\mathfrak{0011}]} \setminus \overline{\Omega}_{[\mathfrak{0111}]}$} (f);
\draw (b) edge[bend left={+70},style={->,dashed},looseness=2.5] node[above, xshift=0cm]  
{$\Omega_{[\mathfrak{0001}]} \setminus \overline{\Omega}_{[\mathfrak{0111}]} 
  \cup \overline{\Omega}_{[\mathfrak{01}]} \cup \overline{\Omega}_{[\mathfrak{0011}]}$} (f);
\end{tikzpicture}
\end{center}
\end{minipage}
\caption{This diagram depicts the trichromatic (left)
and quadrichromatic (right) wave connections predicted by the theory in {\S}\ref{sec:mcr}. 
A solid edge from $\gw_-$ to $\gw_+$ indicates that 
waves of type $[\gw_- \to \gw_+]$   exist 
for $(a,d) \in \Omega_{[\gw_+]} \cap \Omega_{[\gw_-]}$
with $d > 0$
by applying Theorem \ref{thm:mcr:waves} with option (a).
A dashed edge indicates that this connection is covered
by option (b). In this case, the connection 
is only guaranteed to exist in the parameter region displayed
next to the edge, because intermediate stable roots have 
to be ruled out. The asymptotics in 
{\S}\ref{sec:tri:eq} and {\S}\ref{sec:qdc:eq} guarantee that 
all the relevant parameter regions
are non-empty. In addition, the connections
$[\mathfrak{0001} \to \mathfrak{1101}]$, $[\mathfrak{0001} \to \mathfrak{0111}]$ and $[\mathfrak{0001} \to \mathfrak{1011}]$ 
are not listed because $\Omega_{[\mathfrak{0001}]} \cap \Omega_{[\mathfrak{0111}]}$ is contained in $\Omega_{[\mathfrak{0011}]}$. To prevent clutter,
we have also left out the monochromatic $[\mathfrak{0} \to \mathfrak{1}]$ connection
in the diagram on the right. The words that are gray are \textit{not}
Lyndon words and as such do not have any outgoing arrows;
see the discussion in {\S}\ref{sec:mcr:equiv}.
}
\label{fig:diag:tri:quad:conn}
\end{figure}
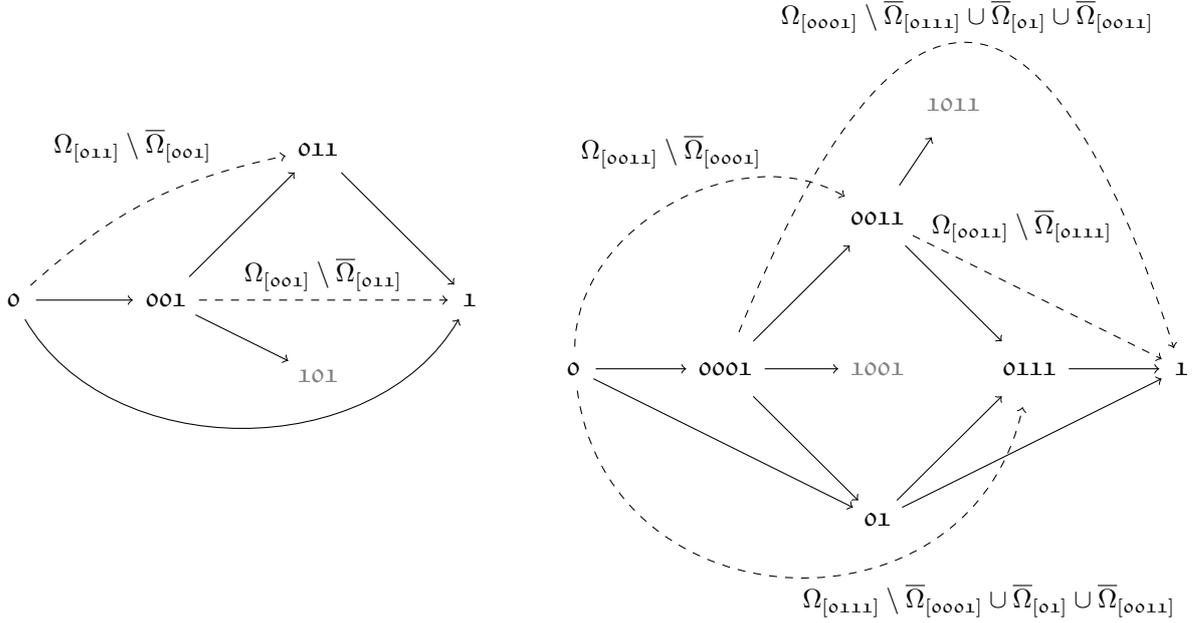

We recall that Theorem \ref{thm:mcr:waves} does not provide any information about the speed $c$ of the travelling wave. We therefore resorted
to numerics to find parameter values where $c \neq 0$ for the waves
discussed in the diagram above. This was done by connecting
the two endstates with a $\tanh$ profile and letting
this initial profile evolve under the flow of \sref{eq:mcr:nagumo:lde}. 
Exploiting the stability of the moving waves, one can test
whether $c = 0$ by determining whether movement ceases after an initial transient period.

For the $[\mathfrak{001} \to \mathfrak{011}]$
and $[\mathfrak{001} \to \mathfrak{101}]$
connections we were not able to find any  regions where $c\neq 0$.
However, in Figure \ref{tric.eq:Speed_thres} we can observe the numerically computed minimum threshold for $d$  where we in fact have $c \neq 0$
for the waves that connect to and from the homogeneous states
$[\mathfrak{0}]$ and $[\mathfrak{1}]$. Notice that
both types of waves have non-zero speed in the region around
the fold and cusp points of $\Gamma_{[\mathfrak{001}]}$. This indicates
that travelling waves can appear and disappear twice 
as the diffusion coefficient is increased,
which does not happen in the bichromatic case.

In Figure \ref{tric.fig:LDE_C_D} we give snapshots
of the collision process that occurs as
a $[\mathfrak{0} \to \mathfrak{001}]$ wave collides with a 
$[\mathfrak{001} \to \mathfrak{1}]$ wave.
In both cases an intermediate buffer-zone consisting
of the trichromatic state $[\mathfrak{001}]$ is consumed
by one (or both) of the incoming trichromatic waves,
leading eventually to a pinned monochromatic wave.
This type of collision can also be observed in the bichromatic setting.

\begin{figure}[t]
\begin{center}
\includegraphics[width=\textwidth]{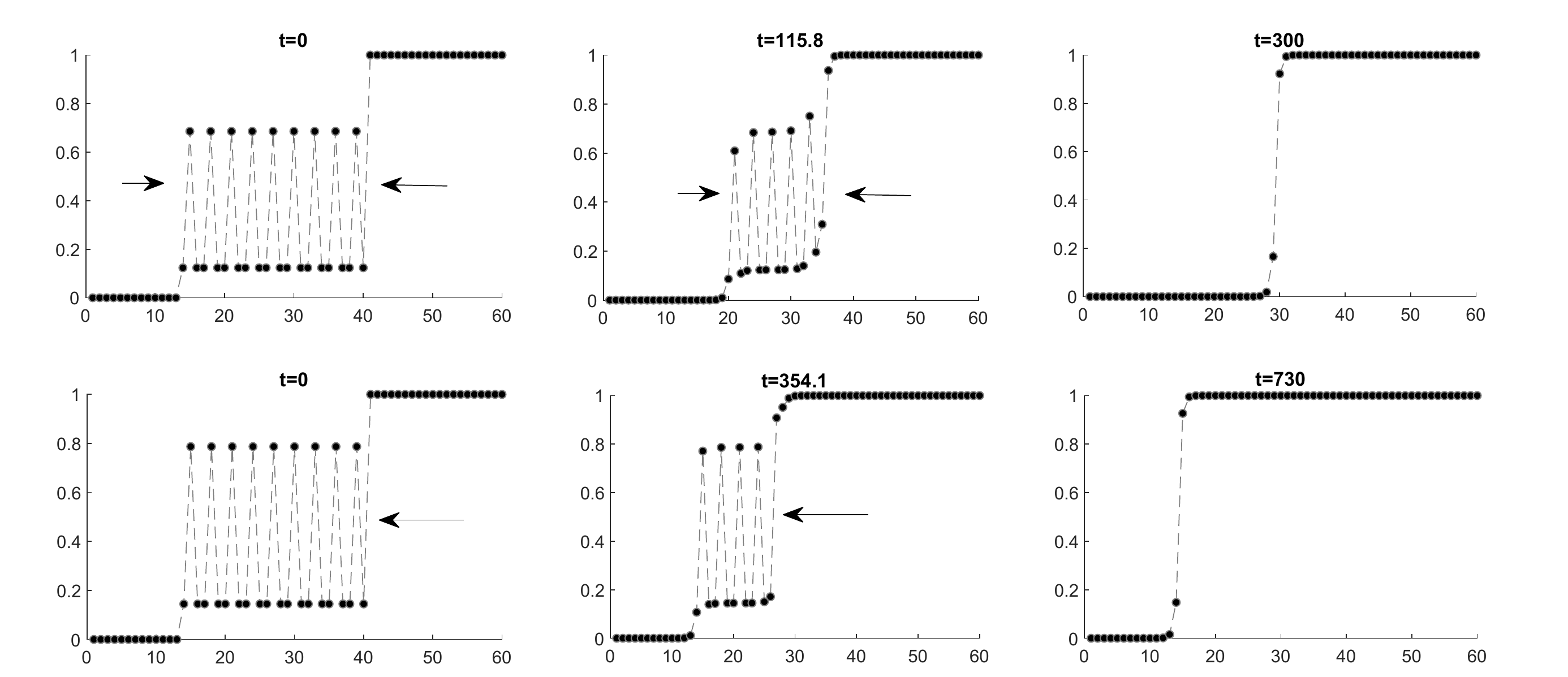}
\caption{
These panels describe two simulations of the 
LDE \sref{eq:mcr:nagumo:lde} that feature
a collision between a $[\mathfrak{0} \to \mathfrak{001}]$ wave
and a $[\mathfrak{001} \to \mathfrak{1}]$ wave.
For the top three panels we have
$a=0.404$ and $d=0.054$, which is above both trichromatic speed
thresholds.  However, it is still below the monochromatic
threshold. In particular, the buffer zone is now consumed
from both sides, but the end result is again a pinned
monochromatic front.
For the bottom three panels we have
$a=0.404$ and $d=0.05$,
which is now below the $[\mathfrak{0} \to \mathfrak{001}]$
but above the $[\mathfrak{001} \to \mathfrak{1}]$
speed thresholds. Here the $[\mathfrak{001}]$ buffer zone is consumed
from the right, resulting in a pinned monochromatic front.
} 
\label{tric.fig:LDE_C_D}
\end{center}
\end{figure}


\section{Quadrichromatic waves}
\label{sec:quad}

In this section we discuss the quadrichromatic case $n = 4$. As in the trichromatic case, 
stable $4$-periodic equilibria can disappear and reappear as $d$ is increased for a fixed $a$. The novel behaviour in this setting is that 
\emph{travelling} quadrichromatic waves can 
co-exist with \emph{travelling} 
monochromatic waves for an open set of parameters $(a,d)$.
This allows several new types of collisions to occur. For example,
two incoming connections with intermediate quadrichromatic states
can collide to form a monochromatic travelling wave.

\subsection{Equilibria}
\label{sec:qdc:eq}

The relevant nonlinearity that governs quadrichromatic
equilibria to \sref{eq:mcr:nagumo:lde} is now given by
\begin{equation}
\label{eq:qdr:def:G}
G(\gu;a,d) := \begin{pmatrix}
d (\gu_4-2\gu_1+\gu_2) + g\big(\gu_1; a\big) \\
d (\gu_1-2\gu_2+\gu_3) + g\big(\gu_2; a\big) \\
d (\gu_2-2\gu_3+\gu_4) + g\big(\gu_3; a\big) \\
d (\gu_{3}-2\gu_4+\gu_1) + g\big(\gu_4; a\big) \\
\end{pmatrix}.
\end{equation}
Inspecting this system shows that one component
can be removed 
by enforcing either $\gu_1 = \gu_3$
or $\gu_2 = \gu_4$. Of course, this problem reduces to the
bichromatic case $n=2$ if both these identities are enforced.
On the other hand, if one takes
\begin{equation}
\label{eq:qdc:red:aa:bb}
    \gu_1 = \gu_2 = \gu_A,
    \qquad \qquad
    \gu_3 = \gu_4 = \gu_B,
\end{equation}
the system $G(\gu; a, d) = 0$ reduces to
\begin{equation}
  \label{eq:qdr:red:sys:aa:bb:case}
    d( \gu_B - \gu_A ) + g( \gu_A ; a) = d ( \gu_A - \gu_B) + g(\gu_B ; a) = 0.
\end{equation}
This again corresponds to the bichromatic case $n =2$ 
but now with the halved diffusion 
coefficient $d_{\mathrm{\mathrm{bc}}} = \frac{1}{2} d$.

In Figure \ref{quad.fig:root_thres} we display several
numerically computed curves $\Gamma_{[\gw]}$ in the critical set $\Gamma$ that correspond
with the upper boundaries of the sets $\Omega_{[\gw]}$.
For visual clarity, we only consider the words
\begin{equation}
\gw \in \{ \mathfrak{0001},  \mathfrak{0011},
\mathfrak{01},\mathfrak{0111} \} ,
\end{equation}
which correspond with the stable non-homogeneous equilibria for $G(\gu; a, d) = 0$. 

The curves $\Gamma_{[\mathfrak{0001}]}$ and 
$\Gamma_{[\mathfrak{0111}]}$ again contain cusp and fold points. 
However, all four curves can \emph{locally} be described 
as a graph $d = d_{[\gw]}(a)$
near the corner $(a,d) = (0,0)$. We now set out
to compute the first two terms 
in the asymptotic expansion of each of these curves.

\begin{figure}[t]
\begin{center}
\includegraphics[width=1\textwidth]{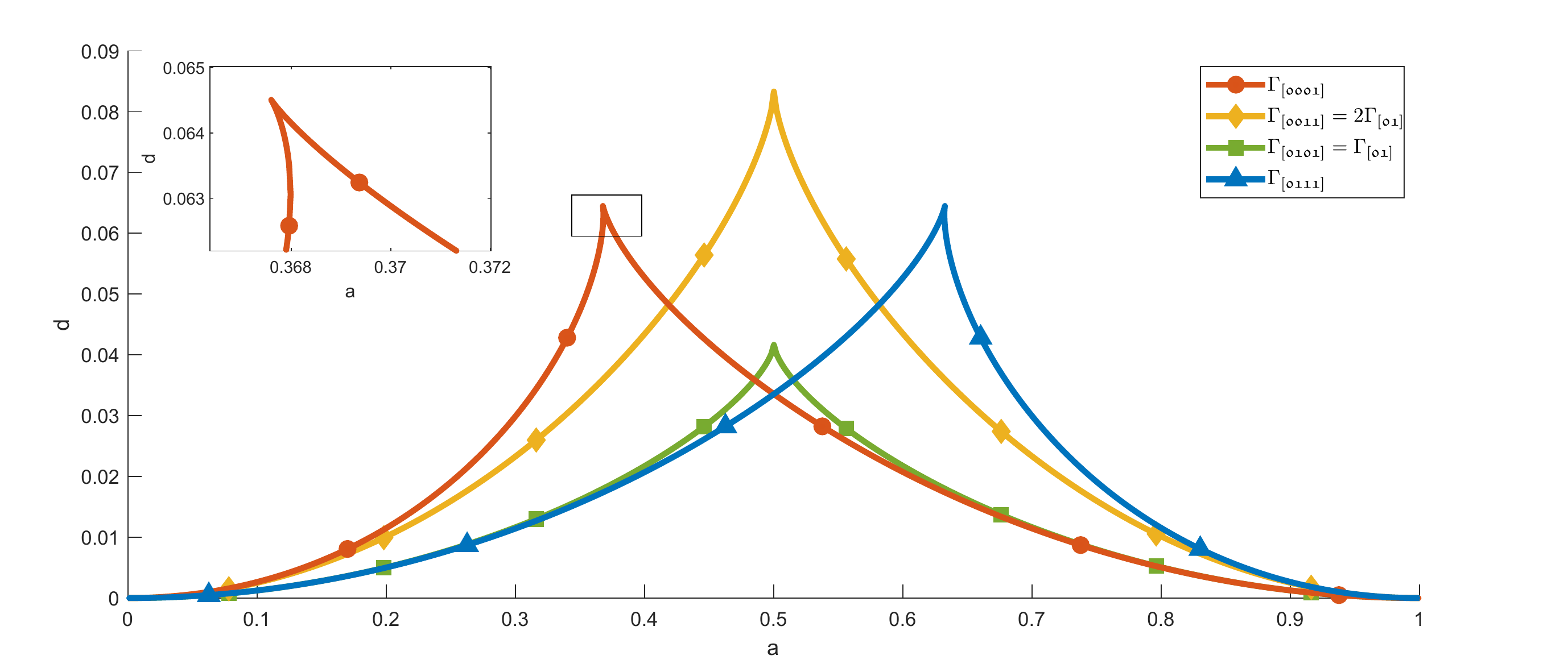}
\caption{Bifurcation thresholds for the four stable quadrichromatic 
root classes. The $\Gamma_{[\mathfrak{0001}]}$
and $\Gamma_{[\mathfrak{0111}]}$ curves again feature slanted cusps;
see the inset.} \label{quad.fig:root_thres}
\end{center}
\end{figure}

\begin{figure}[t]
\begin{center}
\includegraphics[width=1\textwidth]{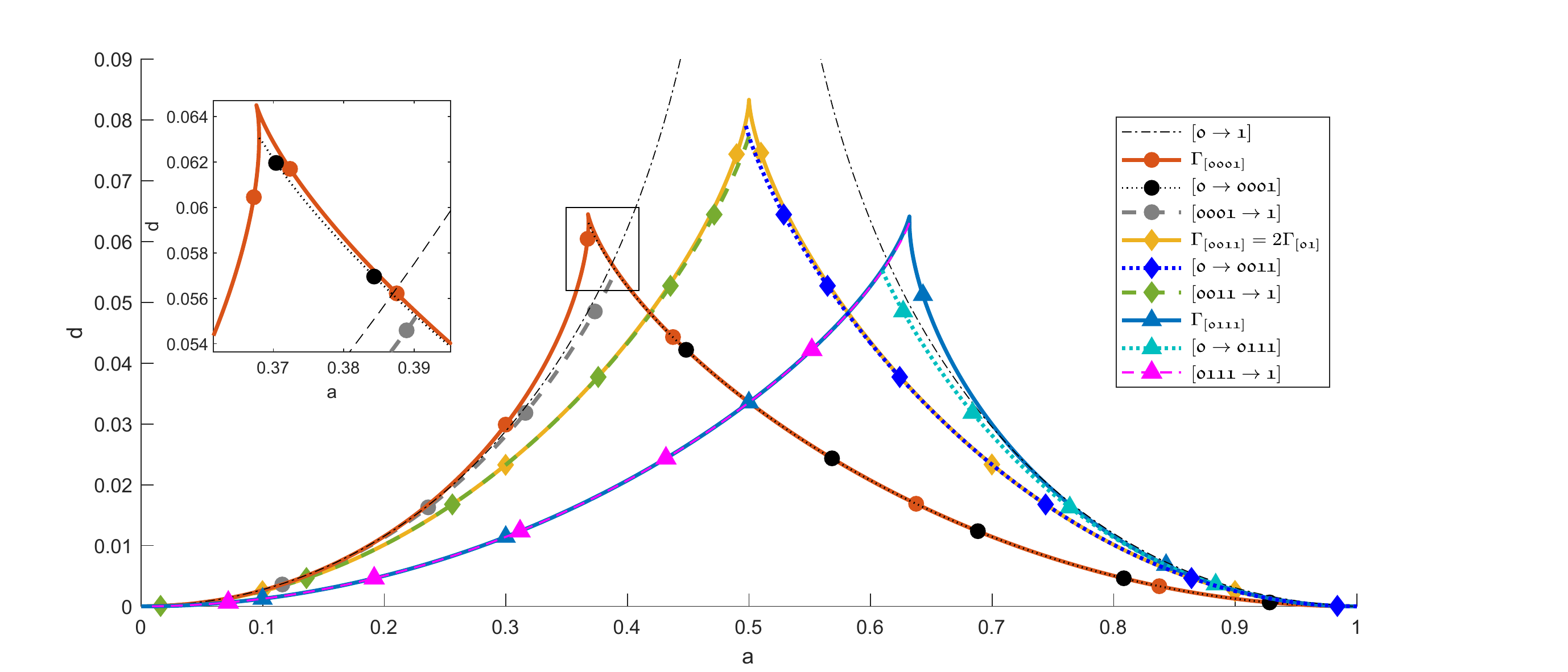}
\caption{Speed thresholds for the wave connections
to and from the stable roots 
$[\mathfrak{0001}]$, $[\mathfrak{0011}]$ and $[\mathfrak{0111}]$. 
The thresholds for $[\mathfrak{0101}]$ are not shown, as they coincide with the bichromatic 
results obtained in \cite{HJHBICHROM}.
} 
\label{quad.fig:speed_thres_complete}
\end{center}
\end{figure}

\begin{figure}[t]
\begin{center}
\includegraphics[width=1\textwidth]{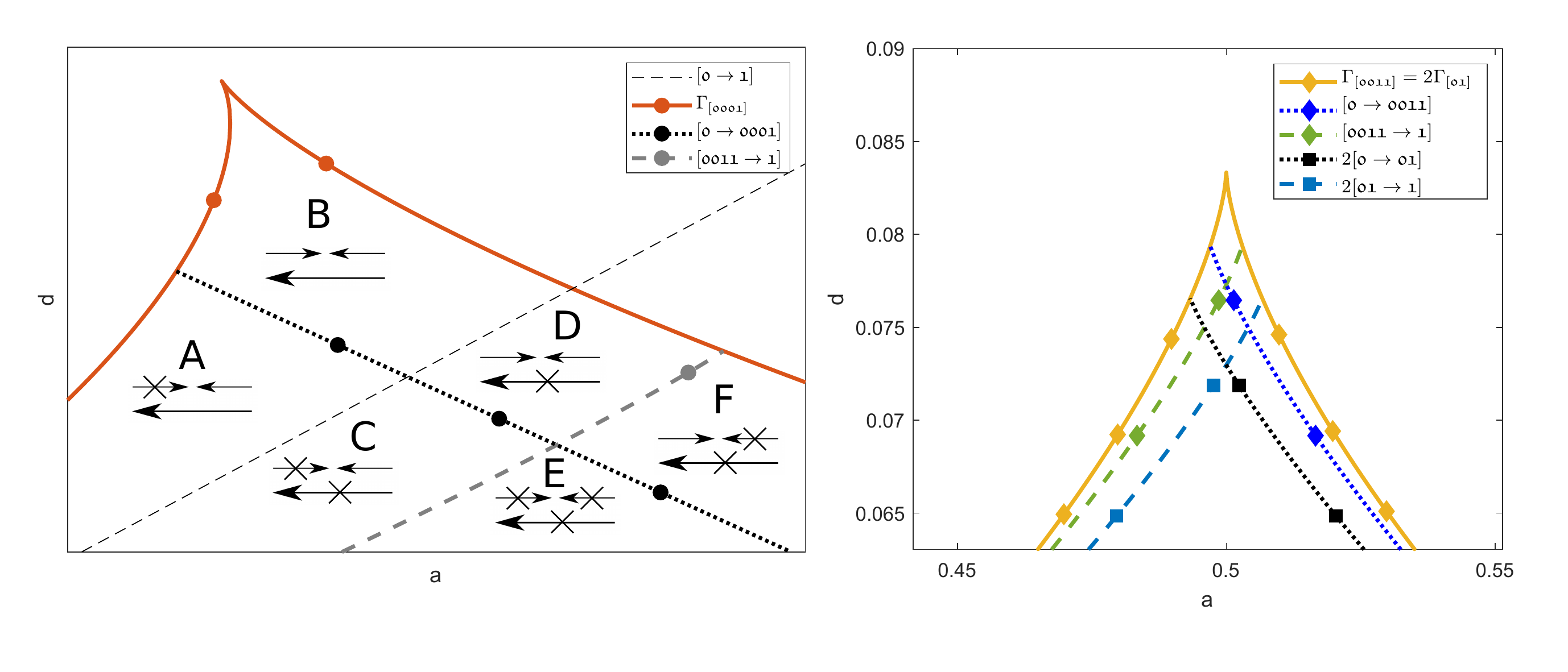}
\caption{The left panel
contains a schematic exaggeration of the area around
the cusp of $\Gamma_{[\mathfrak{0001}]}$,
together with a description of the collision type
that can be expected in each region. 
The top two arrows indicate the nature
of the $[\mathfrak{0} \to \mathfrak{0001}]$
and $[\mathfrak{0001} \to \mathfrak{1}]$ connections
that move towards each other. The bottom arrow
describes whether or not the resulting monochromatic 
wave is pinned.
The right panel compares
the speed threshold for the connections involving
$[\mathfrak{0011}]$ with the doubled
bichromatic thresholds. One can see
that the reduction \sref{eq:qdc:red:aa:bb}
respects the equilibrium structure
but not the wave structure of the system.
} 
\label{quad.fig:speed_thres_zoom}
\end{center}
\end{figure}

\begin{figure}[t]
\begin{center}
\includegraphics[width=1\textwidth]{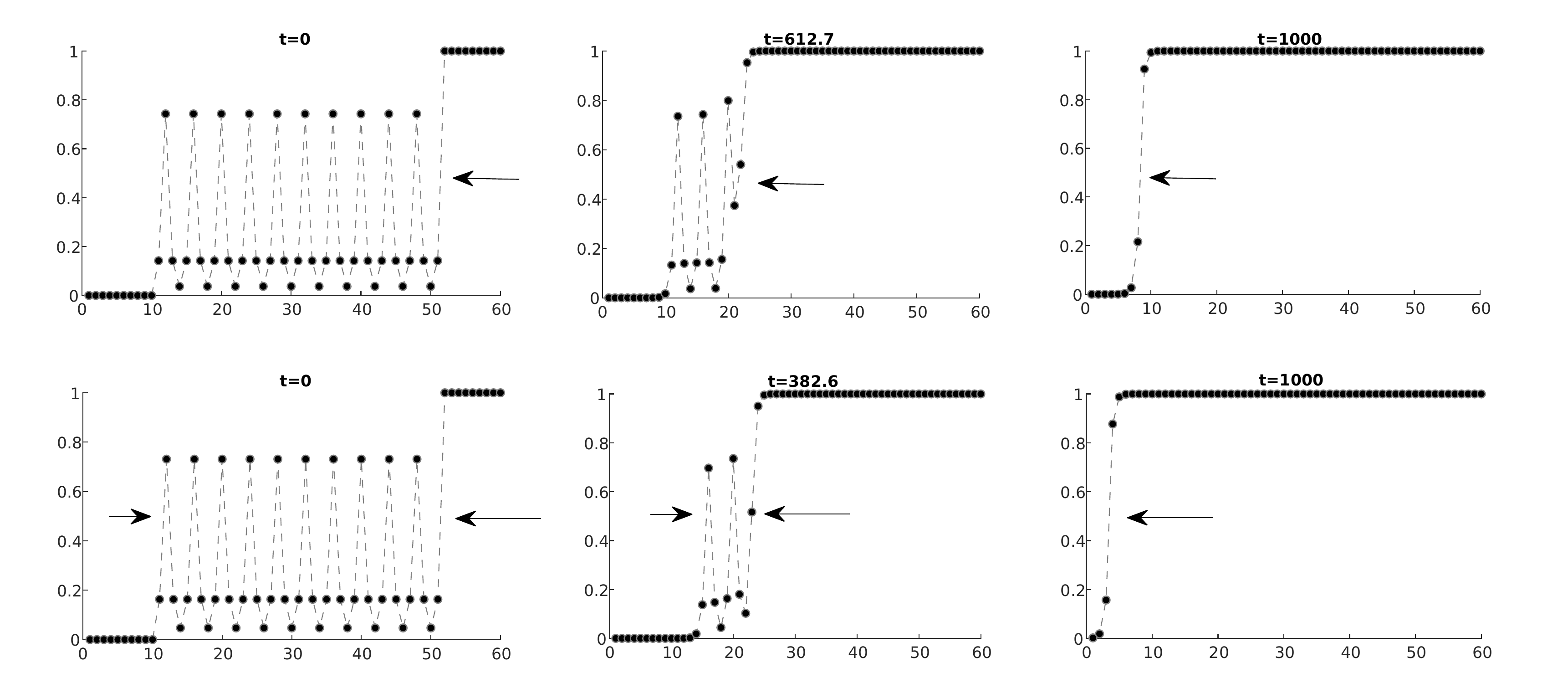}
\caption{
These panels describe two simulations of the 
LDE \sref{eq:mcr:nagumo:lde} that feature
a collision between a $[\mathfrak{0} \to \mathfrak{0001}]$ wave
and a $[\mathfrak{0001} \to \mathfrak{1}]$ wave.
The top three panels feature a collision
of type $A$, with $a = 0.378$ and $d = 0.058$.
The quadrichromatic buffer zone is invaded
from the right, forming a travelling monochromatic
wave once it is extinguished.
The bottom three panels
feature a collision of type $B$,
with $a=0.37$ and $d=0.0625$.
Here the left part of the buffer zone
is first pulled towards zero by the 
incoming wave on the left,
but then gets pulled towards one by the
final travelling monochromatic wave.
} \label{quad.fig:LDE_A_B}
\end{center}
\end{figure}

\paragraph{The $\Gamma_{[\mathfrak{0111}]}$ threshold}
In view of the symmetries discussed above we consider solutions where
the second and fourth component are equal. In particular,
we consider the problem
\begin{equation}
    G\big( ( x , 1 + y, 1 + z, 1 + y) ; a , d \big) = 0 ,
\end{equation}
which can be written as
\begin{equation}
    (H_1, H_2, H_3)( x , y, z ; a,d ) = 0
\end{equation}
with
\begin{equation}
\begin{array}{lcl}
H_{1}(x,y,z;a,d) & =  & 2d(1+y-x) + x(1-x)(x-a), \\[0.2cm]
H_{2}(x,y,z;a,d) & =  & d(x+z -2y -1) - y(1+y)(y+1-a), \\[0.2cm]
H_{3}(x,y,z;a,d) & =  & 2d(y-z) - z(1+z)(z+1-a).    
\end{array}
\end{equation}
Notice that $H_2$ and $H_3$ feature terms of order $O(y)$ respectively $O(z)$,
which corresponds with the fact that the root $g(1 ;a) = 0$ is simple when 
$a = 0$. In addition, $H_3$ is independent of $x$. Setting
$H_3 = 0$ hence allows us to write $z = z_*(y; a, d)$,
which can be substituted into $H_2 =0 $ to yield
$y = y_*( x ; a, d)$. Plugging these expressions
into $H_1$ by writing
\begin{equation}
\tilde{H}_1(x ; a, d) = H_1\big(x, y_*(x; a, d), z_*(y; a, d); a, d\big),
\end{equation}
we find
\begin{equation}
\tilde{H}_1(x; a, d) = x^2 + O\big( x^3 +  a x  + d  \big).
\end{equation}
This allows us to uncover
the saddle-node bifurcations in a fashion analogous
to \cite{HJHBICHROM}.
In particular, we obtain the expansion
\begin{equation}
d_{[\mathfrak{0111}]}(a)= \frac{a^2}{8} +\frac{a^4}{64} +O(a^5),
\end{equation}
together with
\begin{equation}
    \gu_{\mathfrak{0111}}\big(a ; d_{[\mathfrak{0111}]}(a) \big) 
 = \big( \frac{a}{2},
 1-\frac{a^2}{8} -\frac{a^3}{16},
 1 , 
 1-\frac{a^2}{8} -\frac{a^3}{16}
 \big) + O( a^4).
\end{equation}

\paragraph{The $\Gamma_{[\mathfrak{0101}]} = 
  \Gamma_{[\mathfrak{01}]}$ threshold}
The discussion above implies that this threshold is identical
to the corresponding threshold for the bichromatic case $n = 2$.
We can hence copy the results from \cite[Prop 3.6]{HJHBICHROM}
and write
\begin{equation}
d_{[\mathfrak{01}]}(a)= \frac{a^2}{8} +\frac{a^4}{32} +O(a^5),
\end{equation}
together with
\begin{equation}
\label{eq:qdr:g:u:01}
    \gu_{\mathfrak{01}}\big(a ; d_{[\mathfrak{01}]}(a) \big) 
 = \big( \frac{a}{2},
 1-\frac{a^2}{4} -\frac{a^3}{8}
  \big) + O( a^4).
\end{equation}

\paragraph{The $\Gamma_{[\mathfrak{0011}]}$ threshold}
The identity \sref{eq:qdr:red:sys:aa:bb:case}
allows us to write
\begin{equation}
d_{[\mathfrak{0011}]}(a) = 2 d_{[\mathfrak{01}]}(a)= \frac{a^2}{4} +\frac{a^4}{16} +O(a^5).
\end{equation}
In addition, we can reuse
the expressions \sref{eq:qdr:g:u:01}
to find
\begin{equation}
  \gu_{\mathfrak{0011}}\big(a ; d_{[\mathfrak{0011}]}(a) \big)
  = \big( \frac{a}{2}, \frac{a}{2},
 1-\frac{a^2}{4} -\frac{a^3}{8},
 1-\frac{a^2}{4} -\frac{a^3}{8}
  \big) + O( a^4).
\end{equation}

\paragraph{The $\Gamma_{[\mathfrak{0001}]}$ threshold}
The symmetries discussed above allow us to consider solutions where
the first and third component are equal. In particular,
we consider the problem
\begin{equation}
    G\big( ( x , y, x, 1 + z) ; a , d \big) = 0 ,
\end{equation}
which can be written as
\begin{equation}
    (H_1, H_2, H_3)( x , y, z ; a,d ) = 0
\end{equation}
with
\begin{equation}
\begin{array}{lcl}
H_{1}(x,y,z;a,d) & =  & d(y+z+1-2x) + x(1-x)(x-a), \\[0.2cm]
H_{2}(x,y,z;a,d) & =  & 2d(x-y) +y(1-y)(y-a), \\[0.2cm]
H_{3}(x,y,z;a,d) & =  & 2d(x-z-1) - z(1+z)(z+1-a).    
\end{array}
\end{equation}
Notice that $H_3$ features a term of order $O(z)$
and is independent of $y$. In fact,
it is very similar to \cite[Eq. (3.30)]{HJHBICHROM},
which allows us to write
\begin{equation}
 z = z_*(x;a,d)=-2d -2 ad + 2d x + O ( d^2 + d a^2 ) .
\end{equation}
However $H_2$ features 
$O\big((a+d)y\big)$ terms, which prevents
us from expressing $y$ in terms of $x$ as before.
This corresponds with the fact that the root $g(0;a) = 0$ is double
at $a = 0$. On the other hand,
setting $H_2 = 0$ and introducing the scalings
\begin{equation}
y =  a \tilde{y},
\qquad
d = d_*(a,\tilde{d}) = \frac{1}{4}a^2( 1 + \tilde{d})
\end{equation} 
does allow us to write
\begin{equation}
  x = x_*(\tilde{y}; a, \tilde{d}) = 
  a \tilde{y}- \frac{2 \tilde{y}}{1 + 4\tilde{d}} 
    (1 - a \tilde{y})(\tilde{y} - 1).
\end{equation}
Substituting these expressions into $H_1$,
we write
\begin{equation}
\tilde{H}_1( \tilde{y} ; a, \tilde{d} )
 = H_1\Big( x_*\big(\tilde{y} ; a, \tilde{d}\big) , a \tilde{y},
 z_*\big(x_*(\tilde{y} ; a, \tilde{d}\big); a, d_*(a , \tilde{d}) \big) ;
   a ,  d_*(a , \tilde{d}) \Big)
\end{equation}
and find
\begin{equation}
\tilde{H}_1(\tilde{y}; a, \tilde{d} ) = 
  4 \tilde{y}^2 + O \big( \tilde{y}^3 +  (a +\tilde{d}) \tilde{y}^2 +  a \tilde{y} +  a^2 \big).
\end{equation}
The saddle-node bifurcations can now be unfolded by examining the
terms in this equation using the procedure in \cite{HJHBICHROM}.
In particular, we find
\begin{equation}
d_{[\mathfrak{0001}]}(a) =\frac{a^2}{4} +\frac{a^3}{8}+O(a^4),
\end{equation}
together with
\begin{equation}
    \gu_{\mathfrak{0001}}\big(a ; d_{[\mathfrak{0001}]}(a) \big) 
 = \big( \frac{a}{2} +\frac{a^2}{8} +\frac{3a^3}{16},
 \frac{a^2}{4} +\frac{a^3}{8},
 \frac{a}{2} +\frac{a^2}{8} +\frac{3a^3}{16},
 1-\frac{a^2}{2} -\frac{a^3}{2}
 \big) + O( a^4).
\end{equation}
The main point of interest here is that $d_{[\mathfrak{0001}]}(a)$ contains a cubic term. In fact,
our expansion here agrees with the expansion
of the formula \cite[Eq. (5.1)]{VL28}, which provides
a (non-sharp) upper bound for the values of $d$ where
monochromatic waves are pinned. As indicated in Figure \ref{quad.fig:speed_thres_complete}, our numerical results
confirm that $\Omega_{[\mathfrak{0001}]}$ intersects
the region in $(a,d)$ space where monochromatic waves can
travel.

\subsection{Wave connections}

As in the trichromatic case,
a visual inspection confirms that assumptions
${\rm (H\Omega 1)}$, ${\rm (H\Omega 2)}$ and ${\rm (HS)}$ are satisfied.
The connections predicted by 
Theorem \ref{thm:mcr:waves} are depicted
in Figure \ref{fig:diag:tri:quad:conn}. 
In Figure \ref{quad.fig:speed_thres_complete}
we provide the numerically computed minimal values for $d$
for which the waves connecting to and from the spatially
homogeneous equilibria $[\mathfrak{0}]$ and $[\mathfrak{1}]$
have a non-zero speed. 
The novel feature here is that there is overlap
with the region where the monochromatic $[\mathfrak{0} \to \mathfrak{1}]$
wave has non-zero speed.  This allows
for situations where the end-product of a collision between
two quadrichromatic wave is no longer a pinned monochromatic wave
but in fact a travelling monochromatic wave.


The parameter regions where  various types of collisions  can occur are
described in Figure \ref{quad.fig:speed_thres_zoom}.
Types $C$-$F$ closely resemble those encountered in
the bichromatic and trichromatic cases. Types $A$ and $B$ are new
and indeed feature travelling monochromatic end-products.
Two examples with snapshots of such collisions are provided in 
Figure \ref{quad.fig:LDE_A_B}.

\section{Proof of Theorem \ref{thm:mcr:waves}}
\label{sec:pmr}

Here we provide the proof of our main result,
which allows us to establish the existence of wave
connections between equilibria by 
simply comparing their types.

We first show that a pair of distinct ordered stationary solutions cannot have two equal components.
Throughout this section we assume that expressions
such as $i-1$ or $i+1$ should be evaluated within the modulo arithmetic on indices $\{1,2,\ldots,n\}$.

\begin{lemma}
\label{l:one:coordinate:implies:equatity}
Assume that $\gu,\gv \in \Real^n$ satisfy
$G(\gu; a, d) = G(\gv;a,d) =0$ for some pair 
$a \in (0,1)$ and $d > 0$. Assume furthermore that
$\gu \leq \gv$ and that $\gu_i = \gv_i$ for some 
$i \in \{1, \ldots, n\}$.
Then in fact $\gu=\gv$.
\end{lemma}
\begin{proof}
Since $\gu_i=\gv_i$ we have
\begin{equation}
0=d (\gu_{i-1}-2\gu_i+\gu_{i+1}) + g\big(\gu_{i}; a\big) = d (\gu_{i-1}-2\gv_i+\gu_{i+1}) + g\big(\gv_{i}; a\big) = d (\gv_{i-1}-2\gv_i+\gv_{i+1}) + g\big(\gv_{i}; a\big) ,
\end{equation}
which implies that
\begin{equation}
\gu_{i-1}+\gu_{i+1}= \gv_{i-1}+\gv_{i+1} .
\end{equation}
Since $\gu_{i-1}\leq\gv_{i-1}$ and $\gu_{i+1}\leq\gv_{i+1}$ we obtain $\gu_{i-1}=\gv_{i-1}$ and $\gu_{i+1}=\gv_{i+1}$. 
This argument can subsequently be repeated
a number of times to yield $\gu = \gv$. 
\end{proof}

The main ingredient in our proof of Theorem \ref{thm:mcr:waves}
is that the ordering of any pair of words from the
full set $\{ \mathfrak{0}, \mathfrak{a} , \mathfrak{1} \}^n$
and the stable subset $\{ \mathfrak{0},   \mathfrak{1} \}^n$
is preserved for the equilibria that have the corresponding
types.  For example, for $n = 4$ 
we have the partial ordering
\begin{figure}[H]
\begin{center}
\includegraphics[width=0.8\textwidth]{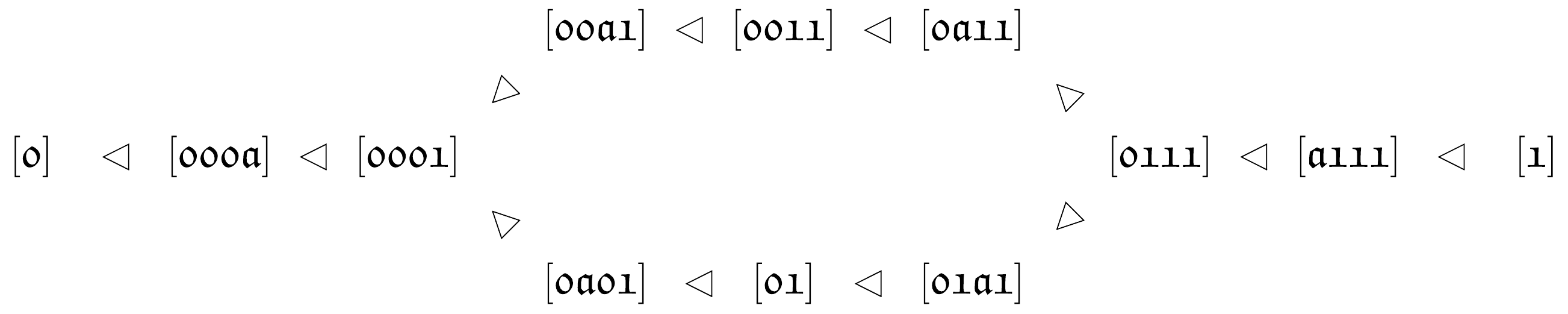}
\end{center}
\end{figure}

\noindent whereby each $[\gw_A] \triangleleft [\gw_B]$ 
connection in this diagram indicates that
$\gu_{\gw_A}(a,d) < \gu_{\gw_B}(a,d)$ whenever
$(a,d) \in \Omega_{\gw_A} \cap \Omega_{\gw_B}$ with $d > 0$.

\begin{lemma}\label{l:ordering}
Assume that ${\rm (H\Omega1)}$ and ${\rm (H\Omega2)}$ are satisfied
and consider a distinct pair $\gw_A,\gw_B \in \{\mathfrak{0},\mathfrak{a},\mathfrak{1}\}^n$ that 
admits the ordering
$\gw_A \le \gw_B$. Suppose furthermore that
at least one of these two words
is contained in $\{\mathfrak{0}, \mathfrak{1} \}^n$.
Then for any $(a,d) \in \Omega_{\gw_A} \cap \Omega_{\gw_B}$
we have the strict component-wise inequality
\begin{equation}
 \label{eq:pmr:ord:u:a:vs:b}
    \gu_{\gw_A}(a,d) < \gu_{\gw_B}(a,d) .
\end{equation}
\end{lemma}

\begin{proof}
Fixing $(a,d) \in \Omega_{\gw_A} \cap \Omega_{\gw_B}$,
we note that ${\rm (H\Omega 1)}$ and ${\rm (H\Omega 2)}$
allow us to pick a curve
\begin{equation}
    [0,1] \ni t \mapsto
      \big( \gv_A(t) , \gv_B(t), \alpha(t), \delta(t) \big) \in 
      [0,1]^n \times [0,1]^n \times (0,1) \times [0, \infty)
\end{equation}
so that we have
\begin{equation}
\begin{array}{lcl}
    (\gv_A, \gv_B, \alpha, \delta)(0) & = & 
     \big( [\gw_A]_{\mid a},
      [\gw_B]_{\mid a} , a, 0 \big), \\[0.2cm]
    (\gv_A, \gv_B, \alpha, \delta)(1) & = & 
    \big( 
      \gu_{\gw_A}(a,d),
      \gu_{\gw_B}(a,d)  , a, d \big), \\[0.2cm]  
\end{array}
\end{equation}
while the inclusion
\begin{equation}
  \big(\alpha(t), \delta(t) \big) \in \Omega_{\gw_A} \cap \Omega_{\gw_B}    
\end{equation}
and the identities
\begin{equation}
G\big( \gv_A(t) ; \alpha(t), \delta(t) \big) = 
G\big( \gv_B(t) ; \alpha(t), \delta(t) \big)
=0
\end{equation}
all hold for $0 \le t \le 1$.
By slightly modifying the path
and picking a small $\epsilon > 0$,
we can also  ensure that $\alpha(t) = a$ and $\delta(t) = t$
for all $t \in [0, \epsilon)$.

Upon introducing the graph Laplacian $B: \Real^n \to \Real^n$
and the nonlinearity $\Psi: \Real^n \to \Real^n$ that act as
%
\begin{equation}
(B \gu )_i = \gu_{i-1}  -2 \gu_i + \gu_{i+1},
\qquad
\Psi( \gu)_i = g( \gu_i; a),
\end{equation}
we see that 
\begin{equation}\label{eq:proof:implicitG}
0 = G (\gv_{\#}(t); a, t) =  t B \gv_{\#}(t) + \Psi\big(\gv_{\#}(t) \big)    
\end{equation}
for $t$ small and $\# \in \{A, B\}$.
Taking implicit derivatives of \eqref{eq:proof:implicitG}, we find
\begin{equation}
\label{eq:pmr:impl:fnc:high:deriv}
t B \frac{d^k}{d t^k} \gv_{\#}(t) + k B \frac{d^{k-1}}{d t^{k-1}} \gv_{\#}(t)
     + \mathcal{R}^{(k)}_{\#}(t)    
      = - D \Psi\big( \gv_{\#}(t) \big) \frac{d^k}{dt^k} \gv_{\#}(t),
\end{equation}
in which we have defined
\begin{equation}
\mathcal{R}^{(k)}_{\#}(t)    
 = \sum_{j=1}^{k-1} {k-1 \choose j-1} 
      \Big[\frac{d^{k - j}}{dt^{k - j}} D \Psi\big( \gv_{\#}(t) \big) \Big]\frac{d^j}{dt^j}
      \gv_{\#}(t)
\end{equation}
for $k \ge 2$, setting this expression to zero for $k = 1$.

For any index $i$ we define
the quantity 
\begin{equation}
\ell_i = \min\{ j' \ge 0: (\gw_{A})_{i + j'} < (\gw_{B})_{i + j'}
\hbox{ or } (\gw_{A})_{i-j'} < (\gw_{B})_{i - j'}
\},
\end{equation}
which measures the distance to the closest index
where $\gw_A$ and $\gw_B$ are unequal.
We now claim that for any $k \ge 0$ we have
\begin{equation}
    \frac{d^k}{d t^k} (\gv_{A})_i(0)
     \le \frac{d^k}{d t^k} (\gv_{B})_i(0)
\end{equation}
if $\ell_i \ge k$, with the inequality being
strict if and only if $\ell_i = k$.
For $k = 0$ this is obvious. Assuming this holds
for $k-1$, consider any index with $\ell_i \ge k \ge 1$.
Our alphabet assumption implies that
\begin{equation}
    (\gw_A)_i = (\gw_B)_i \neq \mathfrak{a},
\end{equation}
which implies that the $ii$-component of the 
two diagonal
matrices $D \Psi\big( \gv_{A}(0) \big)$
and $D \Psi\big( \gv_{B}(0) \big)$
are strictly negative; see \sref{eq:mcr:diag:matrix:d1g}.
In addition, 
our induction hypothesis implies that 
\begin{equation}
    \mathcal{R}_A^{(k)}(0) = \mathcal{R}_B^{(k)}(0).
\end{equation}
By definition, we have
\begin{equation}
\ell_{i \pm 1} \ge \ell_{i} - 1 \ge k- 1.
\end{equation}
In addition, we have $\ell_i = k$ if and only if
 $\ell_{i+1} = k-1$ or $\ell_{i-1} = k- 1$ holds.
Our induction hypothesis hence implies
\begin{equation}
    \big(B \frac{d^{k-1}}{d t^{k-1}} \gv_A(0) \big)_i
    \le \big( B \frac{d^{k-1}}{d t^{k-1}} \gv_B(0) \big)_i,
\end{equation}
with strict inequality if and only if $\ell_i = k$.
Our claim now follows immediately
from \sref{eq:pmr:impl:fnc:high:deriv}.

The argument above shows that $\gv_A(t) < \gv_B(t)$
for all $t \in (0, \epsilon)$. If \sref{eq:pmr:ord:u:a:vs:b} fails to hold, this hence means
that there exists $t_* \in [\epsilon, 1]$
for which $\gv_A(t_*) \le \gv_B(t_*)$,
with also $\big(\gv_A(t_*)\big)_i = \big(\gv_B(t_*)\big)_i$
for some $i \in \{1, \ldots, n\}$.
Lemma \ref{l:one:coordinate:implies:equatity}
now implies $\gv_A(t_*) = \gv_B(t_*)$
and hence 
\begin{equation}
    \gu_{\gw_A}\big( \alpha(t_*), \delta(t_*) \big)
    = \gu_{\gw_B}\big( \alpha(t_*), \delta(t_*) \big),
\end{equation}
which violates Corollary \ref{cor:mcr:unq:props}.
\end{proof}

By combining Lemma's \ref{l:one:coordinate:implies:equatity} and \ref{l:ordering} we can control all the (marginally) stable equilibria
in the box $[\gu_{\gw_-}, \gu_{\gw_+}]$. This allows us to finally prove our main result.

\begin{proof}[Proof of Theorem \ref{thm:mcr:waves}]
Pick any distinct pair $\gw_\pm \in \{ \mathfrak{0}, \mathfrak{1} \}^n$ with $\gw_- \le  \gw_+$
and any $(a,d) \in \Omega_{\gw_-} \cap \Omega_{\gw_+}$
with $d > 0$. Lemma \ref{l:ordering}
implies that the cuboid 
\begin{equation}
    \mathcal{K} = \{ \gu \in \Real^n: 
    \gu_{\gw_-}(a,d) \le \gu \le \gu_{\gw_+}(a,d) \} 
    \subset [0,1]^n
\end{equation}
has non-empty
volume. In addition, for any $\gw \in \{ \mathfrak{0}, \mathfrak{1} \}^n$ that does not satisfy 
$\gw_- \le \gw \le \gw_+$
and for which $(a,d) \in \overline{\Omega}_{\gw}$,
we have $u_{\gw}(a,d) \notin \mathcal{K}$
on account of Lemma \ref{l:one:coordinate:implies:equatity},
the connectedness of $\Omega_{\gw_\pm} \cap \Omega_{\gw}$
and continuity considerations.

In view of (HS), all equilibria inside 
the cube $\mathcal{K}$ besides the two corner points $\gu_{\gw_\pm}(a,d)$
hence have a strictly positive eigenvalue.
The existence of the wave $(c,\Phi)$ now follows
from \cite[Thm. 6]{CHENGUOWU2008}.
\end{proof}

\bibliographystyle{klunumHJ}

\end{document}